\documentclass[11pt]{article}

\usepackage[T1]{fontenc}
\usepackage{lmodern}
\usepackage[margin=1in]{geometry}
\usepackage{microtype}
\usepackage{amsthm,amsmath,amsfonts,amssymb,mathtools}
\usepackage[numbers,sort&compress]{natbib}
\usepackage[colorlinks,citecolor=blue,urlcolor=blue,linkcolor=blue,
  bookmarksdepth=3]{hyperref}
\usepackage{array}
\usepackage{enumitem}
\usepackage{graphicx}
\usepackage{tikz}
\usetikzlibrary{arrows.meta,positioning}
\mathtoolsset{showonlyrefs=true} 

\newtheorem{theorem}{Theorem}[section]
\newtheorem{lemma}[theorem]{Lemma}
\newtheorem{proposition}[theorem]{Proposition}
\newtheorem{corollary}[theorem]{Corollary}

\theoremstyle{definition}
\newtheorem{remark}{Remark}[section]

\newenvironment{HRproblem}[1]
{\par\medskip\noindent\textbf{HR Problem #1.}\ \ignorespaces}
{\par\medskip}

\newcommand{\R}{\mathbb{R}}

\newcommand{\norm}[1]{\left\lVert #1 \right\rVert}

\newcommand{\X}{\mathcal{X}}
\newcommand{\F}{\mathcal{F}}

\newcommand{\E}{\mathbb{E}}
\newcommand{\Prob}{\mathbb{P}}
\newcommand{\Var}{\operatorname{Var}}

\newcommand{\Normal}{\mathsf{N}}

\newcommand{\one}{\mathbf{1}}
\newcommand{\Fcoef}[2]{\widehat{#1}(#2)}
\newcommand{\Fs}{\mathcal{S}}
\newcommand{\Fsum}[2]{(\Fs_{#1}#2)(0)}
\newcommand{\Fmean}[2]{(\mathcal{M}_{#1}#2)(0)}
\newcommand{\Fm}{\mathcal{M}}

\newcommand{\B}{E}

\tikzset{
  mcstate/.style={
    draw=black!70,
    fill=black!4,
    rounded corners=2pt,
    minimum height=7mm,
    inner xsep=5pt,
    font=\small
  },
  mchub/.style={
    circle,
    draw=black!85,
    fill=black!12,
    minimum size=9mm,
    inner sep=0pt,
    font=\small
  },
  mcedge/.style={-{Latex[length=2mm]},thick,draw=black!80},
  mcfactorone/.style={-{Latex[length=2mm]},thick,draw=black!85},
  mcfactortwo/.style={
    -{Latex[length=2mm]},thick,densely dashed,draw=black!70
  },
  mcproduct/.style={-{Latex[length=2mm]},very thick,draw=black},
  mcprob/.style={font=\scriptsize,fill=white,inner sep=1.5pt},
  mcvalue/.style={font=\scriptsize,text=black!75,align=center}
}

\begin{document}

\centerline{\LARGE\bf Markov Chain CLTs: Resolving Open Problems}
\bigskip
\centerline{\large Austin Brown$^a$, Jeffrey S.\ Rosenthal$^b$, and Quan Zhou$^{a,}$\footnote{Corresponding author. Email: \texttt{quan@stat.tamu.edu}}}
\bigskip
\centerline{\large $^a$Texas A\&M University \ and \ $^b$University of Toronto}
\bigskip

\begin{abstract}
Markov chain central limit theorems (CLTs) and their associated variances are very important for implementing Markov chain Monte Carlo  algorithms among other applications. 
H\"aggstr\"om and Rosenthal (2007) presented various results regarding the equality of different formulae for this variance and also posed seven open problems. We resolve all seven in this paper.  
Our main result is a general characterization of Markov chain CLTs: for stationary and ergodic Markov chains, we prove that the normalized partial sums satisfy a $\sqrt{n}$-CLT if and only if the Poisson equation admits a finite solution with square-integrable martingale differences whose second moment equals the CLT variance. In particular, this result does not require reversibility of the chain or square integrability of the function. 
It implies that Roberts' holding-probability condition always precludes a $\sqrt{n}$-CLT, and for reversible chains,  whenever a $\sqrt{n}$-CLT holds, the variance limit is finite if the function is square-integrable, otherwise undefined.  
In addition, we develop a general principle that expresses Fourier coefficients as the autocovariances of an ergodic nonreversible Markov chain. 
\end{abstract}



\bigskip

\section{Introduction} 

\subsection{Background on Markov Chain CLTs}
The central limit theorem (CLT) for i.i.d.\ finite-variance random variables is well known. For additive functionals of Markov chains, similar CLTs also exist, and the theory has been developed along several complementary lines. Spectral methods yield CLTs for reversible chains~\citep{KipnisVaradhan1986, DerriennicLin2001}; 
martingale-based methods treat broader classes of chains through $L^2$ bounds on conditional expectations~\citep{MaxwellWoodroofe2000, WuWoodroofe2004, CunyPeligrad2012}; drift, mixing, and regeneration methods provide verifiable conditions on general state spaces~\citep{Jones2004, RobertsRosenthal2004, bednorz:etal:2008};
and more recent works have also investigated criteria for chain-wide CLTs and heavy-tailed stationary distributions~\citep{CunyLin2025, BresarMijatovicRoberts2025}. 
Such CLT results are particularly useful for obtaining error bounds in Markov chain Monte Carlo  algorithms~\citep{Geyer1992, JonesEtAl2006, FlegalJones2010}, which often require estimating the asymptotic variance $\sigma^2$ associated with the corresponding CLT. 

\citet{HaggstromRosenthal2007} studied three variance formulae associated with a centered function $h$ of a stationary Markov chain $(X_n)_{n \geq 0}$. In their notation, 
\begin{equation*}\label{eq:intro-A-B}
    S_n \coloneqq \sum_{k=0}^{n - 1} h (X_k), \qquad  
    A \coloneqq \lim_{n \rightarrow \infty } \frac{1}{n}\Var_\pi\left( S_n \right), \qquad 
    B \coloneqq  \gamma_0+2\sum_{k=1}^\infty \gamma_k, 
\end{equation*}
where $\gamma_k=\E_\pi\bigl[h(X_0)h(X_k)\bigr]$ is the lag-$k$ autocovariance. 
For reversible chains with square-integrable $h$, let $C$ denote the spectral representation of the asymptotic variance.  Formal definitions   are given in Section~\ref{sec:three-var}. 
Building on the classical result of~\citet{KipnisVaradhan1986} on the relation between a $\sqrt{n}$-CLT and finiteness of $C$, H\"aggstr\"om and Rosenthal established conditions under which the quantities $A, B, C$ are all equal to each other and to the variance $\sigma^2$ appearing in the CLT: $S_n / \sqrt{n} \Rightarrow \Normal(0, \sigma^2)$. 
The behavior of $A$ has also been studied in many other works. For example, chain-level conditions ensuring finiteness of $A$ are identified in~\citet{roberts:rosenthal:2008} and~\citet{DerriennicLin2011}, spectral characterizations are developed in~\citet{deligiannidis:etal:2015} for normal (not necessarily reversible) transition operators,  variational and Poisson-equation formulas for nonreversible chains are given in~\citet{HuangMao2023}, and sharper results are available for time-sampled chains, independence Metropolis--Hastings algorithms, alternating kernels and Gibbs samplers~\citep{LatuszynskiRoberts2013, maire:douc:olsson:2014, deligiannidis:etal:2018, qin:2024}.  

Against this background, \citet[Section 5]{HaggstromRosenthal2007} posed seven open problems (henceforth HR Problems) regarding the behavior of $A, B, C$ under more general conditions.  
To our knowledge, those problems have remained unresolved, and this paper presents solutions to all of them.

\subsection{Main Results and Solutions to HR Problems}
Among the seven problems, HR Problems 2--6 share a common theme. While much of the related literature focuses on sufficient conditions for Markov chain CLTs, these open problems ask what can be deduced from the existence of a $\sqrt{n}$-CLT. This converse perspective motivates our main result, Theorem~\ref{th:poisson}, which provides a general framework for addressing such problems in a unified way. 
For stationary and ergodic chains with $h \in L^1(\pi)$, we show that  a $\sqrt{n}$-CLT holds if and only if there exists a finite-valued function $u$ solving the Poisson equation $u - P u = h$, and its associated martingale difference, 
\begin{equation}\label{eq:intro-mart-diff}
    d(X_0, X_1) \coloneqq u(X_1) - (Pu)(X_0),
\end{equation} 
is square-integrable. When such $u$ exists, the second moment of $d(X_0, X_1)$ equals the CLT variance $\sigma^2$.  
To our knowledge, this result provides the first converse, under only $h\in L^1(\pi)$, to existing approaches that use Poisson equations to establish  Markov chain CLTs~\citep{MaxwellWoodroofe2000, WuWoodroofe2004, RobertsRosenthal2004}.
We also prove that the existence of a $\sqrt{n}$-CLT is equivalent to the tightness of  the sequence  $(S_n/\sqrt{n})_{n \geq 1}$, which extends the same dichotomy result of~\citet[Theorem II.3.1]{Chen1999} beyond the setting $h \in L^2(\pi)$. 
Neither result requires reversibility.

An application of Theorem~\ref{th:poisson} shows that the divergence condition introduced by~\citet{Roberts1999} involving holding probabilities cannot coexist  with a $\sqrt{n}$-CLT, which resolves HR Problems~3 and~5; see Theorem~\ref{thm:holding-noise-bound} and Corollary~\ref{coro:HR3and5}. 
Roberts' condition has been used in the literature to establish the failure of a $\sqrt{n}$-CLT for independence Metropolis--Hastings samplers~\citep{Roberts1999}, although a complete proof has not previously appeared. 
Our result fills this gap without requiring reversibility of the chain or square-integrability of $h$. 
The proof involves deriving a lower bound on $\sigma^2$ using the martingale difference given by~\eqref{eq:intro-mart-diff}, which satisfies $\E_\pi[ d(X_0, X_1)^2 ] = \sigma^2$; our bound extends existing formulae for reversible chains with square-integrable $h$~\citep{DoucetEtAl2015, deligiannidis:etal:2018}.  

For reversible chains,  we can combine Theorem~\ref{th:poisson} with the forward--backward martingale argument~\citep{Wu1999} by studying the paired martingale difference $d(X_0, X_1) + d(X_1, X_0)$. 
This enables us to prove that $A$ must be finite or undefined for a $\sqrt{n}$-CLT to hold (or equivalently, for the sequence $(S_n/\sqrt{n})_{n \geq 1}$ to be tight), resolving HR Problems~2,~4 and~6; see Theorem~\ref{thm:poisson-reversible} and Corollary~\ref{coro:P2}. 
While many existing theories and examples pointed toward this conclusion~\citep{DoukhanMassartRio1994, HaggstromRosenthal2007, zhao:woodroofe:2010, longla:peligrad:2012, bradley:2026}, it had not been previously proved. 
When $h \in L^2(\pi)$, our additional conclusion is that tightness forces $A = B = C < \infty$; combined with classical results, this shows that $A < \infty$ is both necessary and sufficient for a $\sqrt{n}$-CLT. 
When the CLT holds but $h \notin L^2(\pi)$, we prove that $A, B$ must be undefined: $A$ is undefined since the variance of $S_n/\sqrt{n}$ has oscillatory behavior with limit inferior $\sigma^2$ and limit superior $\infty$, and $B$ is undefined since lagged autocovariances (which exist) alternate between $\infty$ and $-\infty$. Thus, the qualitative behavior exhibited by Example~14 of~\citet{HaggstromRosenthal2007} is unavoidable. (And, $C$ cannot even be defined when $h\not\in L^2(\pi)$.) 
The key insight obtained by our proof is that, for reversible chains, tightness of $(S_n/\sqrt{n})_{n\geq 1}$ implies square-integrability of $h(X_0) + h(X_1)$, even when $h \notin L^2(\pi)$; in other words, the non-square-integrable contribution must cancel across adjacent observations.  
 
The solutions to HR Problems~1 and~7 rely on different techniques and do not follow from Theorem~\ref{th:poisson}. 
We resolve both problems by constructing counterexamples; see Theorems~\ref{thm:hr1-counterexample} and~\ref{thm:hr7-counterexample}. 
In particular, HR Problem~1 concerns the behavior of $A, B$ for nonreversible chains,  
and our solution provides an explicit example where a $\sqrt{n}$-CLT holds with $h \in L^2(\pi)$, $A < \infty$, but $B$ is undefined. 
This adds to the existing literature demonstrating that the CLT behavior can be considerably less structured without reversibility~\citep{Haggstrom2005, HaggstromRosenthal2007}.  
To construct the example, we develop a general Fourier-to-Markov covariance realization principle in Theorem~\ref{thm:representation}: from any strictly positive, continuous, even, $2\pi$-periodic function $f$, we can construct a countable-state, irreducible, aperiodic, positive recurrent Markov chain with a centered $h \in L^2(\pi)$ whose lagged autocovariances coincide with the Fourier coefficients of $f$. 
Thus, the covariance behavior of nonreversible Markov chains is rich enough to encode the Fourier coefficients of every function in this broad class. 

\subsection{Organization of the Paper}
In Section~\ref{sec:main}, we present our main result on the Poisson-equation characterization of the $\sqrt{n}$-CLT for a stationary and ergodic Markov chain. The resulting lower bound on $\sigma^2$ is derived in Section~\ref{sec:bound-sigma}, and the consequences   for reversible chains are presented in Section~\ref{sec:rev-poisson}. 
In Section~\ref{sec:HR}, we formally introduce the three variance expressions $A, B, C$ and present the seven HR problems. 
Solution to HR Problems~2,~4 and~6 is presented in Section~\ref{sec:P2}, that to HR Problems~3 and~5 is presented in Section~\ref{sec:P3}, and that to HR Problems~1 and~7 is presented in Section~\ref{sec:P1}. 
The proof of Theorem~\ref{thm:hr1-counterexample} (solution to HR Problem~1) is deferred to Section~\ref{sec:proof-P1}, where we also present the general Fourier-to-Markov representation theorem. 
Section~\ref{sec:disc} concludes the paper with a brief discussion.

\subsection{Statement on the Use of AI} 
The authors used ChatGPT-5.6 Sol to assist with the proof development and figure generation. The proof arguments were derived through extensive back-and-forth interaction with ChatGPT, which gradually led the authors to formulate the theorems/lemmas and their complete proofs.  In particular, the authors proposed the Poisson-equation characterization in the main result Theorem~\ref{th:poisson}, and the solution to HR Problem~2 involved a lengthy discussion and investigation into counterexamples and techniques  with ChatGPT over the course of a week. 
All of the proofs were completely verified, evaluated and re-written entirely by the human authors, who take full responsibility for the accuracy, correctness and completeness of all content in the paper.

\section{A General Characterization of Markov Chain CLTs} \label{sec:main} 
\subsection{Notation and Definitions} \label{sec:notation}
Throughout this paper, we let $P$ be a time-homogeneous Markov transition kernel on a measurable space $(\X, \F)$, with stationary probability distribution $\pi$. Denote the $n$-step transition kernel by $P^n$ and the holding probability at $x$ by $r(x) = P(x, \{x\})$. 
We assume that $\X$ is a Polish space and $\F$ is the Borel sigma-algebra so that the function $r(x)$ is measurable and $\F$ is countably generated.  
We further assume that $P$ is  ergodic, which means that it is $\phi$-irreducible and aperiodic. 
It follows that $P$ is also recurrent~\citep[Proposition 10.1.1]{meyn:tweedie:2009} and $\pi$-irreducible~\citep[Theorem 10.4.9]{meyn:tweedie:2009}, and $P^n(x, \cdot)$ converges to $\pi$ in total variation distance for $\pi$-almost every $x$~\citep[Theorem 4]{RobertsRosenthal2004}. 
Moreover, there exists an absorbing set $G$ with $\pi(G) = 1$ such that $P$ restricted to $G$ is positive Harris recurrent~\citep[Theorem 9.0.1]{meyn:tweedie:2009}. 

For every measurable function $f : \X \to \R$ and every probability measure $\mu$ on  $(\X, \F)$, define 
\begin{align*}
(P^nf)( x ) = \int_{\X} f(y) P^n(x, dy),  \qquad (\mu P^n) ( \B ) = \int_{\X} P^n(x, \B ) \mu(dx),
\end{align*}  
for every $x \in \X$ and $\B \in \F$. 
We also write $\pi(f) = \int_{\X} f(x) \pi(d x)$. 
Let $\E_{\mu}, \Var_\mu$ denote  the expectation and variance under initial distribution $\mu$, and we write $\E_x = \E_{\delta_x}$.  
Let $\Gamma$ denote the probability measure on $\X \times \X$ defined by 
\begin{equation}\label{eq:def-Gamma}
    \Gamma(dx, dy) = \pi(d x) P(x, dy). 
\end{equation}

We always use $(X_n)_{n \geq 0}$ to denote a stationary Markov chain with kernel $P$ so that for each $n \geq 0$, $X_n \sim \pi$ and $(X_n, X_{n + 1}) \sim \Gamma$. 
Fix some $h \in L^1(\pi)$. 
Let $\Rightarrow $ denote convergence in distribution, and we  say that a $\sqrt{n}$-central limit theorem (CLT) holds if, for some $\sigma^2 \in [0, \infty)$,
\begin{equation}\label{eq:partial-sum-definition} 
\frac{S_n}{\sqrt n} \Longrightarrow \Normal(0,\sigma^2),
\text{ where } S_n \coloneqq \sum_{k=0}^{n-1}h(X_k).  
\end{equation}    
We will use $\sigma^2$ to denote this asymptotic variance throughout. 
Note that the existence of a $\sqrt{n}$-CLT implies $\pi(h) = 0$. 
One standard approach to establishing Markov chain CLTs is to find the solution to the Poisson equation $u - P u = h$. As shown in~\citet[Theorem~32]{RobertsRosenthal2004}, if there exists $u \in L^2(\pi)$ such that $u - P u = h$, then  $d(X_n, X_{n+1}) \coloneqq u(X_{n+1}) - (Pu)(X_n)$ is a stationary, ergodic and square-integrable martingale difference sequence. By the martingale CLT (see Lemma~\ref{lem:mart-CLT}) and the decomposition 
\begin{equation}\label{eq:decomp-Sn-mart-diff}
    S_n = \sum_{k=0}^{n-1} d(X_k, X_{k + 1}) + u(X_0) - u(X_n), 
\end{equation}
we have $S_n/\sqrt{n} \Rightarrow \Normal (0, \sigma^2)$ with $\sigma^2 = \|d\|_{L^2(\Gamma)}^2$. 

We say a family of real-valued random variables $(Y_i)_{i \in \mathcal{I}}$, where $\mathcal{I}$ is an arbitrary index set, is tight,  if $\sup_{i \in \mathcal{I}} \Prob( |Y_i| \ge N) \downarrow 0$ as $N \rightarrow \infty$.   
If $S_n/\sqrt{n} $ converges in distribution or $\sup_{n \geq 1} n^{-1}\E_\pi[S_n^2] < \infty$,  then the sequence $(S_n/\sqrt{n})_{n \geq 1}$ is tight. 

\subsection{Main Theorem} \label{sec:first-main}
 
Our main theorem provides a new tool for deriving necessary conditions for Markov chain CLTs using martingale differences associated with the Poisson equation. 

\begin{theorem}\label{th:poisson}
    Let $P$ be ergodic and $(X_n)_{n \geq 0}$ be stationary. 
    Given $h \in L^1(\pi)$,  define $S_n = \sum_{k=0}^{n-1}h(X_k)$.
    The following statements are equivalent. 
    \begin{enumerate}[label=(\roman*)]
        \item $(S_n / \sqrt{n})_{n \geq 1}$ is tight.
        \item $S_n / \sqrt{n} \Rightarrow \Normal(0,\sigma^2) $ for some $0 \leq \sigma^2 < \infty$. 
        \item There exists a  finite, real-valued and measurable   $u$ such that $(I - P)u = h$ and $P |u|  < \infty$  $\pi$-almost everywhere, and  
        \begin{equation}\label{eq:def-mart-diff}
            d(x, y) = u(y) - (P u)(x) \in L^2(\Gamma). 
        \end{equation} 
    \end{enumerate}
    Whenever these statements hold, $\pi(h) = 0$, and for every $u$ satisfying (iii), $\| d \|^2_{L^2(\Gamma)} = \sigma^2$, 
    \begin{equation*}
         \E\left[  d(X_0, X_1) \mid X_0 \right] = 0 \text{ $\pi$-a.e.}, \qquad 
        h(x) = d(x, y) + u(x) - u (y)  \text{ $\Gamma$-a.e.}. 
    \end{equation*} 
\end{theorem}

\begin{remark}
Existing theories often rely on square-integrability assumptions of $h$ and then use the martingale argument to establish the CLT~\citep{Wu1999, MaxwellWoodroofe2000, WuWoodroofe2004, longla:2017}; in particular, the implication \((iii)\Rightarrow(ii)\) follows from the standard martingale CLT argument described in Section~\ref{sec:notation}. 
So the main step of our proof is the converse implication \((i)\Rightarrow(iii)\), which uses only $h \in L^1(\pi)$ and tightness  of $S_n/\sqrt{n}$ to recover the martingale difference sequence $d(X_n, X_{n+1})$.    
The key innovation of our proof is to construct $u$ as the limit of a sequence of bounded approximate Poisson solutions using ``characteristic-function resolvents'' associated with partial sums stopped at an independent geometric time.  This sequence has uniform $L^2$ control over the associated martingale differences. 
To justify the approximation and pass to an exact solution, we use the Poisson solution operator derived in~\citet{Neveu1972}, which maps each integrable, centered function to a normalized, finite-valued solution of the corresponding Poisson equation; see Lemma~\ref{lem:neveu-inverse} below. This argument does not require reversibility and allows the limiting Poisson solution $u$ to lie outside \(L^1(\pi)\). 
Related characteristic resolvent constructions have appeared in other contexts, for example in the large deviation theory~\citep{kontoyiannis:meyn:2003}.  
\end{remark}

\begin{lemma}\label{lem:neveu-inverse}
There exist a measurable function $a : \X \to (0, 1]$, a probability measure $\rho(dx) = a(x)\pi(dx)/\pi(a)$,   and a linear map
$\mathcal U : L^1_0(\pi) \longrightarrow L^1(a\pi)$  
such that, for every $f \in L^1_0(\pi)$,
\begin{equation}\label{eq:neveu-poisson-inverse}
    P|\mathcal U f| < \infty, \qquad
    (I-P)(\mathcal U f )= f \quad \text{$\pi$-almost surely}, \qquad
    \rho(\mathcal U f) = 0.
\end{equation}
Moreover,
\begin{equation}\label{eq:neveu-weighted-bound}
    \norm{\mathcal U f}_{L^1(a\pi)}
    \leq (1 + \pi(a)^{-1})\norm{f}_{L^1(\pi)},
\end{equation}
and, for every $v \in L^1(\pi)$,
\begin{equation}\label{eq:neveu-normalization}
    \mathcal U [(I -P) v] = v - \rho(v).
\end{equation}
In particular, if $f_j \to f$ in $L^1(\pi)$ with $f_j, f \in L^1_0(\pi)$, then $\mathcal U f_j \to \mathcal U f$ in $\pi$-probability.
\end{lemma}
\begin{proof}
    See Section~\ref{appx:neveu} in the Appendices. 
\end{proof}

\begin{proof}[Proof of Theorem~\ref{th:poisson}]
We first show that (i) implies (iii). 
For $0 < p < 1$, let $N_p$  be independent of $(X_n)_{n \ge 0}$ with  $\Prob( N_p = n ) = p(1-p)^{n- 1}$ for each $n \geq 1.$ 
All expectations below include the randomness of $N_p$. 
Assuming that $(S_n / \sqrt{n})_{n \geq 1}$ is tight, by Lemma~\ref{lem:random-index}, we can choose $t \neq 0$ such that 
\begin{equation}\label{eq:choice-t}
    \sup_{0 < p < 1} \left| \E_\pi[ e^{i t \sqrt{p} S_{N_p} } ]  - 1 \right| \leq \frac{1}{2}. 
\end{equation}
Then, define 
\begin{equation*}\label{eq:def-cp}
 \theta = t \sqrt{p}, \qquad g_p(x) = \E_x\left[ \exp(i \theta S_{N_p}) \right], \qquad c_p = \pi(g_p). 
\end{equation*} 
We have $|g_p| \leq 1$ and, by~\eqref{eq:choice-t}, $1/2 \leq | c_p | \leq 1$ for every $p$. 
Define real-valued bounded functions 
\begin{equation}\label{eq:def-up}
    u_p = \operatorname{Re} \frac{g_p}{i \theta c_p}, \qquad 
    d_p(x, y) = u_p(y) - (P u_p)(x), \qquad
    h_p = (I - P) u_p. 
\end{equation}
Clearly, $d_p$ satisfies $\E[ d_p(X_0, X_1) \mid X_0] = 0$ $\pi$-almost surely; that is, $d_p(X_n, X_{n+1})$ is a martingale difference sequence. 
And since $\pi P = \pi$ and $u_p$ is bounded,  $\pi(h_p) = 0$. 
We defer the proof of the following two lemmas to the next subsection. 
\begin{lemma}\label{lem:hp-L1-conv}
As $p \downarrow 0$,   $h_p \rightarrow h$ in $L^1(\pi)$.
\end{lemma}
\begin{lemma}\label{lem:dp-bound}
  $\sup_{0 < p < 1} \|d_p\|_{L^2(\Gamma)}^2 \leq 4 / t^2 $.
\end{lemma}

Now let $\mathcal{U}$ denote the normalized Poisson inverse map as defined in  Lemma~\ref{lem:neveu-inverse}, and define $u = \mathcal{U} h $. By Lemma~\ref{lem:neveu-inverse}, $u$ and $\mathcal{U} h_p$ are solutions to corresponding Poisson equations such that  
\begin{equation}\label{eq:neveu-hp}
P |u| < \infty, \; (I - P)u = h \text{ $\pi$-almost surely, and }
    \mathcal{U} h_p = u_p - \rho(u_p), 
\end{equation}  
where  $\rho$ is another probability measure on $\X$ equivalent to $\pi$.  
Moreover, since $h_p \rightarrow h$ in $L^1(\pi)$, Lemma~\ref{lem:neveu-inverse} also yields that $\mathcal{U} h_p \rightarrow u$ in $\pi$-probability. 
Hence,  by~\eqref{eq:neveu-hp} and $h_p = (I - P) u_p$,  
\begin{align*}
    d_p(x, y)     &= (\mathcal{U} h_p)(y) - (\mathcal{U} h_p)(x) + h_p(x) \\
    &\longrightarrow u(y) - u(x) + h(x)  = d(x, y), 
\end{align*}
where $d(x, y) \coloneqq u(y) - (P u)(x)$ and the convergence happens in $\Gamma$-probability as $p \downarrow 0$. 
Choose a subsequence $p_j$ such that $d_{p_j}  \rightarrow d$  $\Gamma$-almost everywhere. Fatou's lemma and Lemma~\ref{lem:dp-bound} yield
\begin{equation*}
 \|  d \|_{L^2(\Gamma)}^2  \leq \liminf_{j \rightarrow \infty } \|  d_{p_j} \|_{L^2(\Gamma)}^2 \leq \frac{4}{t^2} < \infty, 
\end{equation*}
which establishes (iii). Note that Lemma~\ref{lem:hp-L1-conv} also implies $\pi(h) = 0$.  And since $(I - P)u = h$,  $h(x) = d(x, y) + u(x) - u(y)$ $\Gamma$-almost everywhere. 

The proof of (iii) $\Rightarrow$ (ii) is  essentially the same  as that of Theorem 32 in~\citet{RobertsRosenthal2004}, which was described in Section~\ref{sec:notation}. 
With $d$ given by~\eqref{eq:def-mart-diff}, $( d(X_n, X_{n+1}) )_{n \geq 0}$ is a stationary, ergodic and square-integrable martingale difference sequence with respect to the filtration $\F_k = \sigma(X_0, \dots, X_k)$. Decomposing $S_n$ as in~\eqref{eq:decomp-Sn-mart-diff} and applying Lemma~\ref{lem:mart-CLT} yields $S_n/\sqrt{n} \Rightarrow \Normal (0, \sigma^2)$ with $\sigma^2 = \|d\|_{L^2(\Gamma)}^2$.  
In particular, we do not need $u \in L^1(\pi)$. 
Finally, (ii) $\Rightarrow$ (i) by the definitions of weak convergence and tightness. 
\end{proof}

\subsection{Proofs of Lemmas~\ref{lem:hp-L1-conv} and~\ref{lem:dp-bound} }
 
\begin{proof}[Proof of Lemma~\ref{lem:hp-L1-conv}] 
We first derive two identities involving $g_p$. Conditioning on whether $N_p = 1$ and applying the Markov property and memoryless property of $N_p$, we find that 
\begin{equation}\label{eq:g_p_identity} 
g_p(x) = e^{i \theta h(x)} [p + (1 - p)(P g_p)(x) ]. 
\end{equation}  
Rearranging the terms yields 
\begin{equation}\label{eq:g_p_identity_2} 
    (I - P) g_p =  (1 - e^{-i \theta h}) g_p + p(1 - P g_p). 
\end{equation}
We show that $g_p$ is approximately a constant independent of $x$ as $p \downarrow 0$. Since $P$ is an $L^1(\pi)$-contraction, $\|P g_p \|_{L^1(\pi)} \leq 1$, and thus by~\eqref{eq:g_p_identity_2}, as $p \downarrow 0$,  
\begin{equation}\label{eq:L1-gp}
    \| (I - P) g_p\|_{L^1(\pi)} \leq \pi( \left| 1 - e^{i \theta h} \right| ) + 2 p  \rightarrow 0. 
\end{equation}
Let $\mathcal{U}$ be as defined in  Lemma~\ref{lem:neveu-inverse}, which can also be applied to complex-valued centered functions by considering the real and imaginary parts separately. In particular, since $\pi[ (I-P) g_p] = 0$, by~\eqref{eq:L1-gp},~\eqref{eq:neveu-weighted-bound} and~\eqref{eq:neveu-normalization},  
$g_p - \rho(g_p) = \mathcal{U}[ (I-P) g_p ]   \rightarrow 0$  in $\pi$-probability, where $\rho$ is another probability measure and thus $|g_p - \rho(g_p)| \leq 2$. Bounded convergence theorem yields that $g_p - \rho(g_p) \rightarrow 0$ also in $L^1(\pi)$, and consequently, $c_p - \rho(g_p) \rightarrow 0$. Using $|c_p| \geq 1/2$, we thus get 
\begin{equation*}
     \frac{g_p}{c_p} = \frac{g_p}{ \rho(g_p)} \frac{\rho(g_p)}{c_p} \rightarrow 1 \text{ as $p \downarrow 0$,  in $L^1(\pi)$ and in $\pi$-probability.}
\end{equation*}
 
By~\eqref{eq:g_p_identity_2}, we can express $h_p$ as 
\begin{equation}\label{eq:hp-decomp}
    h_p = \operatorname{Re} \frac{(I-P) g_p}{i \theta c_p} 
    = \operatorname{Re} \left[  \frac{1 - e^{-i \theta h}}{i \theta} \frac{g_p}{c_p} + \frac{p (1 - P g_p)}{i \theta c_p}   \right].
\end{equation}
The three factors on the right-hand side satisfy 
\begin{equation*}
     \frac{1 - e^{-i \theta h}}{i \theta} \rightarrow h, \qquad 
     \frac{g_p}{c_p} \rightarrow 1 \text{ in $\pi$-probability,} \qquad 
     \left| \frac{p (1 - P g_p)}{i \theta c_p} \right| \leq \frac{4 \sqrt{p}}{|t|}.  
\end{equation*}
The first claim follows from $\theta = t\sqrt{p}$ and that $\lim_{|z| \rightarrow 0} (e^z - 1)/z = 1$ for any complex $z$. 
The third claim follows from $| P g_p | \leq 1$ and $|c_p| \geq 1/2$.  
Therefore, as $p \downarrow 0$, $h_p \rightarrow h$ in $\pi$-probability. 
 
To show $h_p \rightarrow h$ also in $L^1(\pi)$, it suffices to verify the uniform integrability. For any $y \in \R$, 
\begin{equation*}
    1 - e^{- i y} = i y \int_0^1 e^{-i y s} d s, 
\end{equation*}
and thus $|1 - e^{- i y}| \leq |y|$. Letting $y = \theta h$ and using~\eqref{eq:hp-decomp}, we obtain that 
\begin{align*}
    |h_p| \leq \frac{ \left| 1 - e^{-i \theta h }\right| }{ |\theta| } \frac{ |g_p| }{ |c_p|} + \frac{4 \sqrt{p}}{|t|} \leq 2 |h| + \frac{4}{|t|}, 
\end{align*}
where the right-hand side is integrable and does not depend on $p$. Hence, $(h_p)_{0 < p < 1}$ is a uniformly integrable family, and therefore $h_p \rightarrow h$ in $L^1(\pi)$. 
\end{proof}

\begin{proof}[Proof of Lemma~\ref{lem:dp-bound}]
Fix some $0 < p < 1$. By~\eqref{eq:def-up}, 
\begin{equation*}
    d_p(x, y) = \operatorname{Re} \frac{ g_p(y) - (P g_p)(x) }{ i \theta c_p}. 
\end{equation*}
Since $\E[  g_p(X_1) \mid X_0] =  (P g_p)(X_0)$ and the chain is stationary, 
\begin{align*}
     \E_\pi \left[  |g_p(X_1) -  (P g_p)(X_0) |^2  \right]  &= \E_\pi \left[ \E [ |g_p(X_1)|^2 \mid X_0] -   |(P g_p)(X_0)|^2 \right] \\
     &= \| g_p \|^2_{L^2(\pi)} - \| P g_p \|^2_{L^2(\pi)}. 
\end{align*}
Hence, it follows that 
\begin{equation}\label{eq:dp-upper-bound}
    \| d_p \|_{L^2(\Gamma)}^2 = \E_\pi \left[ | d_p (X_0, X_1)|^2 \right] \leq  \frac{\| g_p \|^2_{L^2(\pi)} - \| P g_p \|^2_{L^2(\pi)}}{ \theta^2 |c_p|^2 }. 
\end{equation}
Using~\eqref{eq:g_p_identity} and applying Jensen's inequality to the convex function $z \mapsto   |z|^2$, we get  
\begin{equation*}
    |g_p|^2 \leq \left| p + (1 - p) (P g_p) \right|^2  \leq p + (1 - p) |P g_p|^2. 
\end{equation*}
Subtracting $|P g_p|^2$ from both sides and integrating over $\pi$ yields 
\begin{equation*}
    \| g_p \|_{L^2(\pi)}^2 - \| P g_p \|_{L^2(\pi)}^2 \leq p  \left( 1 - \| P g_p \|_{L^2(\pi)}^2 \right) \leq p. 
\end{equation*}
The claim then follows from~\eqref{eq:dp-upper-bound}, $\theta = t \sqrt{p}$ and $|c_p| \geq 1/2$. 
\end{proof}

\subsection{A Lower Bound on \texorpdfstring{$\sigma^2$}{sigma2}} \label{sec:bound-sigma}

Recall $r(x) = P(x,\{x\})$ denotes the holding probability.  
Theorem~\ref{th:poisson} can be used to derive a lower bound on $\sigma^2$ using $r$, which is given in~\eqref{eq:holding-energy-bound} below. 
This bound has appeared in the earlier theory for reversible chains with $h \in L^2(\pi)$; see~\citet[Proposition~2]{DoucetEtAl2015} and \citet[Theorem~1]{deligiannidis:etal:2018}. Our approach recovers this sharp inequality through a simpler proof, without requiring $h\in L^2(\pi)$ or reversibility. 
We also note that by Theorem~\ref{th:poisson}, the $\sqrt{n}$-CLT assumption in Theorem~\ref{thm:holding-noise-bound} can be replaced by tightness of $(S_n / \sqrt{n})_{n \geq 1}$.   

\begin{theorem}\label{thm:holding-noise-bound}
Let $P$ be ergodic and $(X_n)_{n \ge 0}$ be stationary. 
Let $h\in L^1(\pi)$. If $S_n/ \sqrt{n} \Rightarrow \Normal(0, \sigma^2)$ for some $\sigma^2 < \infty$, then $h=0$ on $\{r = 1\}$ $\pi$-almost surely, and
\begin{equation}\label{eq:holding-energy-bound}
\int_{\{r < 1 \}}h(x)^2\frac{r(x)}{1 - r(x)}\,\pi(dx) \leq \sigma^2.
\end{equation} 
\end{theorem}

\begin{proof}
Let $u, d$ be as given in Theorem~\ref{th:poisson}. Then,     $h(x) = d(x, y) + u(x) - u(y)$ $\Gamma$-almost everywhere.  
In particular, $h(x) = d(x, x)$ for $\pi$-almost every $x$ with $r(x) > 0$.  
Since $d$ also satisfies  $\int d(x, y) P(x, d y ) = 0$,  we have 
\begin{equation}\label{eq:tmp-rx-hx}
 \int_{ \X \setminus \{x\} } d(x, y) P(x, dy) = - r(x) h(x).  
\end{equation}
Hence, $h(x) = 0$ whenever $r(x) = 1$ holds $\pi$-almost everywhere. If $r(x) < 1$, then
\begin{align*}
    \int d(x, y)^2 P(x, dy) &= r(x) h(x)^2 + \int_{ \X \setminus \{x\} } d(x, y)^2 P(x, dy) \\
    & \geq r(x) h(x)^2 +  \frac{r(x)^2 h(x)^2}{1 - r(x)} = \frac{r(x)}{1 - r(x)} h(x)^2, 
\end{align*}
where the second step follows by applying the Cauchy--Schwarz inequality to~\eqref{eq:tmp-rx-hx}. Integrating over $x$ and using $\| d \|_{L^2(\Gamma)}^2 = \sigma^2$ proves~\eqref{eq:holding-energy-bound}. 
\end{proof}
 
\subsection{Implications of a Reversible CLT} \label{sec:rev-poisson}

For reversible chains, we can combine Theorem~\ref{th:poisson} with the forward-backward martingale decomposition that has been used previously to control $h(X_0) + h(X_1)$ \citep{Wu1999,  zhao:woodroofe:2008, longla:2017}.  This strategy enables us to show that whenever a $\sqrt{n}$-CLT exists, the non-square-integrable contribution must cancel within adjacent pairs, and thus $\Var_\pi(S_n)$ must be well-behaved for every even $n$.  

\begin{theorem}\label{thm:poisson-reversible}
Let $P$ be reversible and ergodic and $(X_n)_{n \ge 0}$ be stationary. 
Let $h\in L^1(\pi)$. If $S_n/ \sqrt{n} \Rightarrow \Normal(0, \sigma^2)$ for some $\sigma^2 < \infty$, then   
\begin{equation}
 \sup_{n \geq 1} \frac{\Var_\pi(S_{2n})}{2n} \leq 2 \sigma^2,  \qquad 
 \lim_{n \rightarrow \infty} \frac{\Var_\pi(S_{2n})}{2n} = \sigma^2. 
 \label{eq:even-bound-main}
\end{equation}  
\end{theorem}

\begin{proof}
By Theorem~\ref{th:poisson}, $\pi(h) = 0$. 
Let $u, d$ be as given in Theorem~\ref{th:poisson} so that $h(x) =  d(x, y) + u(x) - u(y)$. Interchanging $x, y$ and adding the two identities yields 
\begin{equation*}
    h(x) + h(y) = d(x, y) + d(y, x), \text{ $\Gamma$-almost everywhere.}
\end{equation*}
Hence,  the even sum subsequence yields the decomposition
\begin{align*}\label{eq:decomp-S-2n}
S_{2n} &= \sum_{k = 0}^{n-1} \left[ h(X_{2k}) + h(X_{2k + 1}) \right]
= \sqrt{n} (J_n + J^*_n), \\ 
J_n &\coloneqq \frac{1}{\sqrt{n}} \sum_{k = 0}^{n-1} d(X_{2k}, X_{2k + 1}), \qquad J_n^* = \frac{1}{\sqrt{n}} \sum_{k = 0}^{n-1} d(X_{2k+1}, X_{2k }). 
\end{align*}
As in the proof of Theorem~\ref{th:poisson}, $J_n$ is the normalized partial sum of stationary and ergodic martingale differences, and thus $J_n \Rightarrow \Normal(0, \sigma^2)$ with $\sigma^2 = \|d\|_{L^2(\Gamma)}^2$. By the pairwise $L^2$-orthogonality of martingale differences, 
\begin{equation*} \E_\pi \left[  J_n^2 \right] = \frac{1}{n} \sum_{k=0}^{n-1} \E_\pi \left[ d(X_{2k}, X_{2k + 1})^2 \right]  =   \norm{d}_{L^2(\Gamma)}^2 = \sigma^2. 
\end{equation*} 
In particular, $\E_\pi[J_n^2]$ converges to the limiting variance, and 
by~\citet[Theorem 3.6]{billingsley1999convergence},  this implies that $(J_n^2)_{n\geq 1}$ is a uniformly integrable sequence.  
Since the chain is reversible, $J_n^*$ has the same distribution as $J_n$. Hence, $\E_\pi \left[  (J_n^*)^2 \right] = \sigma^2$,  $J_n^* \Rightarrow \Normal(0, \sigma^2)$, and $( (J^*_n)^2 )_{n \geq 1}$ is also uniformly integrable. Therefore, using $(a+b)^2 \leq2 a^2 + 2 b^2$ for $a, b \in \R$, we get 
\begin{align*}
   \Var_\pi( S_{2n}) = \E_\pi[ S_{2n}^2 ] \leq 2n  \left( \E_\pi[ J_n^2  ] + \E_\pi[ (J_n^*)^2  ] \right) \leq 4 n \sigma^2. 
\end{align*}
This proves the inequality in~\eqref{eq:even-bound-main}.  
Further, since  both $(J_n^2)_{n\geq 1}$ and $( (J^*_n)^2 )_{n \geq 1}$ are  uniformly integrable, $(S_{2n}^2/n)_{n\geq 1}$ is also uniformly integrable. Therefore, by the CLT assumption, $\Var_\pi(S_{2n}) / 2n$ must converge to $\sigma^2$. 
\end{proof}

By stationarity, Theorem~\ref{thm:poisson-reversible} implies that, for every $n \geq 0$, $h(X_n) + h(X_{n + 1})$ is square-integrable with variance bounded by $4 \sigma^2$.  
We  give a simple example showing that the factor $4$ in general cannot be improved. 

\begin{proposition}
Let $\X = \{-1, 1\}$ and $P$ be the transition matrix with $P(x, x) = \delta$, $P(x, -x) = 1 - \delta$ for some $0<\delta <1$. Let $h(x) = x$. Then, $\Var_\pi(S_2) = 4 \delta$, while a $\sqrt{n}$-CLT holds with
$\sigma^2 = \delta/(1 - \delta).$
Hence, $\sigma^{-2} \Var_\pi(S_{2}) = 4 (1 - \delta) \rightarrow 4$ as $\delta \downarrow 0$. 
\end{proposition} 
\begin{proof}
The second largest eigenvalue of $P$ is $\lambda = 2\delta - 1$. Hence, the spectral formula yields
$\sigma^2 = (1 + \lambda)/(1 - \lambda) = \delta/(1 - \delta)$.   
And a direct calculation gives $\Var_\pi(S_{2}) = \E_\pi[S_{2}^2 ] = 4 \delta$.  
\end{proof} 

\section{Solutions to HR Problems} \label{sec:HR}
 
\subsection{Three Variance Expressions}\label{sec:three-var}
Let $h \in L^1(\pi)$ with $\pi(h) = 0$.  Denote the lag-$k$ autocovariance by
\begin{equation*}\label{eq:autocovariance-definition}
\gamma_k=\E_\pi\bigl[h(X_0)h(X_k)\bigr],\qquad k\geq 0, 
\end{equation*}
which is allowed to be infinite. When $h \in L^2(\pi)$, $|\gamma_k| < \infty$ for every $k$. 
For $n\geq 1$ and $m\geq 0$, let 
\begin{equation} \label{eq:def-A-B-1}
A_n=\frac{1}{n}\Var_\pi\left( S_n \right),\qquad B_m=\gamma_0+2\sum_{k=1}^m \gamma_k.
\end{equation}
Note that $B_m$ is well-defined provided that each $\gamma_k$ is well-defined and the resulting sum does not contain the form $\infty - \infty$.  
When the limits exist, write 
\begin{equation}\label{eq:A-B-definition}
A=\lim_{n\to\infty}A_n, \qquad  B=\lim_{m \to\infty}B_m, 
\end{equation}
and both limits are allowed to be infinite. We have the elementary identity 
\begin{equation}\label{eq:cesaro-identity}
A_n=\gamma_0+2\sum_{k=1}^{n-1}\left(1-\frac{k}{n}\right)\gamma_k
=\frac{1}{n}\sum_{m=0}^{n-1}B_m,
\end{equation}
provided that $(B_m)_{0 \leq m \leq n-1}$ are well-defined. 
Thus, $A_n$ is the Ces\`aro mean of $B_0, \dots, B_{n-1}$. Consequently, if $B$ exists, then $A$ exists and $A = B$.   
We also record two useful observations about the behavior of $(A_n)_{n \geq 1}$.  

\begin{lemma}\label{lem:inf-A-h2} 
If $A < \infty$,   $h \in L^2(\pi)$. 
If $h \notin L^2(\pi)$,   $\max\{A_n, A_{n+1}\} = \infty$ for every $n \geq 1$.  
\end{lemma}
\begin{proof} 
Since $h(X_n)=S_{n+1}-S_n$ and $(X_n)_{n \geq 0}$ is stationary,  $h \in L^2(\pi)$ whenever both $S_n, S_{n + 1}$ have finite variance for some $n \geq 1$. 
Similarly, if $h \notin L^2(\pi)$, then at least one of $S_n, S_{n + 1}$ has infinite variance for every $n \geq 1$.  
The two assertions then follow.  
\end{proof}

Recall $r(x) = P(x, \{x\})$. 
We say that Roberts' condition holds if
\begin{equation}\label{eq:roberts-cond}
\lim_{n\rightarrow \infty} n \E_\pi\left[h(X_0)^2r(X_0)^n\right]  = \infty. 
\end{equation}  

\begin{lemma}[Roberts, 1999]\label{lem:inf-A-roberts}
    If Roberts' condition holds, then $A = \infty$.
\end{lemma} 
\begin{proof}
See~\citet{Roberts1999}. In particular,  $A_n \geq n \E_\pi\left[h(X_0)^2r(X_0)^n\right]$  for each $n \geq 1$. 
\end{proof}  

When $P$ is reversible with respect to $\pi$ and $h \in L^2(\pi)$, we define the spectral variance of $h$ by
\begin{equation}\label{eq:def-C}
C = \int_{ [ -1, 1] } \frac{1 + \lambda}{1 - \lambda} \mathcal{E}_{h}( d\lambda ),  
\end{equation} 
where $\mathcal{E}_{h}$ is the spectral measure for $h$; that is, $\mathcal{E}_{h}(\cdot) = \langle h,  \mathcal{E}(\cdot) h \rangle_{L^2(\pi)}$  where $\mathcal{E}$ is the unique spectral decomposition measure associated with $P$.  
Under the reversibility and square-integrability assumption, \citet{KipnisVaradhan1986} proved that   $C < \infty$ implies a $\sqrt{n}$-CLT  with $\sigma^2 = C$, and \citet[Theorem 4]{HaggstromRosenthal2007} showed that  $A = B = C \in [0, \infty]$. 
If $P$ is not reversible or $h \notin L^2(\pi)$, we leave $C$ undefined.    

Below we describe the seven open problems posed by~\citet[Section 5]{HaggstromRosenthal2007} about the behavior of $A, B, C$ and present the corresponding solutions we obtain.   

\subsection{HR Problems 2, 4 and 6: Behavior of $A$ under a Reversible CLT}\label{sec:P2}

\begin{HRproblem}{2}
 Is there a reversible, ergodic Markov chain where a $\sqrt{n}$-CLT exists, but $A = \infty$?    
\end{HRproblem}

\begin{HRproblem}{4}
 If the chain is reversible and ergodic and a $\sqrt{n}$-CLT exists, what further conditions (weaker than i.i.d. or uniformly ergodic) still imply that $A < \infty$?     
\end{HRproblem}

\begin{HRproblem}{6}
   If the Markov chain is reversible (and ergodic\footnote{
Throughout, we assume that $P$ is ergodic. Though this  assumption is not explicitly stated in HR Problems~5 and~6 in~\citet{HaggstromRosenthal2007}, it appears to be implicit. 
}) and a $\sqrt{n}$-CLT exists, and $\pi(h^2) < \infty$, does this imply that $A < \infty$? 
\end{HRproblem}
 
For $h \in L^2(\pi)$ with $A = \infty$, several CLT results have been established, but with a normalization different from $\sqrt{n}$~\citep{zhao:woodroofe:2010, longla:peligrad:2012}. 
\citet{bradley:2026} also provided multiple examples where $A = \infty$ but the $\sqrt{n}$-CLT fails.  
For $h \notin L^2(\pi)$, Example 14 of~\citet{HaggstromRosenthal2007} has a $\sqrt{n}$-CLT, while $A$ is undefined rather than infinite~\citep[Remark 15]{HaggstromRosenthal2007}.  
Our Theorem~\ref{thm:poisson-reversible} implies that $A$ can never be infinite when a $\sqrt{n}$-CLT holds for a reversible chain, thus  resolving HR Problem~2 and demonstrating that all these phenomena are not accidental. Moreover, as we prove in Corollary~\ref{coro:P2} below, $A$ must be undefined if a $\sqrt{n}$-CLT exists for $h \notin L^2(\pi)$. 
Theorem~\ref{thm:poisson-reversible} also resolves HR Problem~6, since by Theorem~4 of~\citep{HaggstromRosenthal2007}, $A$ is always defined when $h \in L^2(\pi)$. 

HR Problem~4 is open-ended. Various sufficient conditions  ensuring $A< \infty$ without CLT assumptions exist  in the literature, including the variance-bounding property~\citep{roberts:rosenthal:2008} and spectral criteria~\citep{deligiannidis:etal:2015}.  
Together with Theorem~4 of~\citep{HaggstromRosenthal2007}, our solution to HR Problem~2  implies that $h \in L^2(\pi)$ is a necessary and sufficient condition. 

\begin{corollary}[Solution to HR Problems~2,~4 and~6]\label{coro:P2}
Let $P$ be reversible and ergodic and $(X_n)_{n \ge 0}$ be stationary. 
Let $h \in L^1(\pi)$ be such that $S_n/ \sqrt{n} \Rightarrow \Normal(0, \sigma^2)$ for some $\sigma^2 < \infty$.  
\begin{enumerate}[label=(\roman*)]
    \item If $h \in L^2(\pi)$, then $A = B = C < \infty$. 
    \item If $h \notin L^2(\pi)$, then 
    \begin{equation*}
    \liminf_{n \rightarrow \infty} A_n = \sigma^2, \quad \limsup_{n \rightarrow \infty} A_n = \infty, \quad 
    \gamma_{2k} = \infty \text{ and }\gamma_{2k+1} = -\infty \text{ for } k \geq 0.  
    \end{equation*} 
    \item $A < \infty$ if and only if $h \in L^2(\pi)$. 
\end{enumerate}
\end{corollary}

\begin{proof}
Theorem~\ref{thm:poisson-reversible} shows that $A_{2n} \rightarrow \sigma^2 < \infty$, and thus $A$ cannot be infinite. 
When $h \in L^2(\pi)$, Theorem~4 of~\citep{HaggstromRosenthal2007} ensures that  $A = B = C \in [0, \infty]$ are all well-defined, so they must all be finite. 
This proves part (i). Part (iii) also follows, since by Lemma~\ref{lem:inf-A-h2},  $A < \infty$ implies $h \in L^2(\pi)$. 

Consider part (ii).  We already know from Theorem~\ref{thm:poisson-reversible} that $A_{2n} \rightarrow \sigma^2$ and $A_{2n} \leq 2 \sigma^2$ for each $n$. Then, by Lemma~\ref{lem:inf-A-h2}, $A_{2n + 1} = \infty$ for each $n$, proving $\liminf_n A_n = \sigma^2$ and $\limsup_n A_n = \infty$. 
For the lag-$k$ autocovariance, note that 
\begin{equation}
    h(X_0) - h(X_{2k}) = \sum_{j=0}^{2k-1} (-1)^j [h(X_{j}) + h(X_{j+1})], 
\end{equation}
where the sum of each adjacent pair, $h(X_{j}) + h(X_{j+1})$, is in $L^2$. Hence, $ h(X_0) - h(X_{2k}) \in L^2$. Since for any $a, b \in \R$, the negative part of  $ab$, denoted by $(ab)^{-}$, satisfies $(ab)^{-} \leq (a-b)^2/4$, 
\begin{equation}
    \E_\pi \left[ ( h(X_0) h(X_{2k}) )^- \right] \leq \frac{1}{4} \E_\pi \left[ ( h(X_0) - h(X_{2k}) )^2 \right] 
    < \infty. 
\end{equation}
Meanwhile, because
\begin{equation}
    h(X_0)^2 + h(X_{2k})^2 = ( h(X_0) - h(X_{2k}) )^2  + 2 h(X_0) h(X_{2k}) \text{ is not integrable,} 
\end{equation}
the positive part of $h(X_0) h(X_{2k})$ cannot be integrable, which proves $\gamma_{2k} = \infty$. 
An analogous argument shows that $ h(X_0) + h(X_{2k + 1}) \in L^2$ and consequently $\gamma_{2k + 1} = -\infty$. 
\end{proof}

\subsection{HR Problems 3 and 5: Holding Probabilities and the CLT}
 \label{sec:P3}

\begin{HRproblem}{3}
Is there a reversible, ergodic Markov chain where a $\sqrt{n}$-CLT exists, but where Roberts' condition $\lim_{n \to \infty} n \E\left[ h(X_0)^2 r(X_0)^n \right] = \infty$ holds?
\end{HRproblem}

\begin{HRproblem}{5}
In particular, if the chain is  reversible (and ergodic) and a $\sqrt{n}$-CLT exists, and there is $\delta > 0$ such that $r(x) \geq \delta$ for all $x \in \X$, does this imply that $A < \infty$?      
\end{HRproblem}

We use Theorem~\ref{thm:holding-noise-bound} to show that the answer to HR Problem~3 is negative even without assuming reversibility. 
Together with Theorem~\ref{thm:poisson-reversible}, this also readily implies that the answer to HR Problem~5 is affirmative.  
By Theorem~\ref{th:poisson},  the $\sqrt{n}$-CLT assumption can be equivalently replaced by the tightness of $(S_n / \sqrt{n})_{n \geq 1}$. 

\begin{corollary}[Solution to HR Problems 3 and 5]\label{coro:HR3and5}
Let $P$ be ergodic and $(X_n)_{n \ge 0}$ be stationary.  
Let $h\in L^1(\pi)$ be such that $(S_n / \sqrt{n})_{n \geq 1}$  is tight.  Then, 
\begin{equation}\label{eq:roberts-limit-zero}
\lim_{n\to\infty}n\E_\pi\left[h(X_0)^2r(X_0)^n\right]=0.
\end{equation} 
Further, if $P$ is reversible and $\inf_{x \in \X} r(x) \geq \delta > 0$, then $A  < \infty.$ 
\end{corollary}

\begin{proof}
By Theorem~\ref{th:poisson}, $S_n / \sqrt{n} \Rightarrow \Normal(0, \sigma^2)$ with $\sigma^2 < \infty$. 
Using the elementary inequality
\begin{align}\label{eq:nrn-domination}
nr^n \leq \sum_{k=1}^{n} r^k \leq\frac{r}{1-r}, \qquad 0 \leq r < 1,
\end{align}
we get if $r(X_0) < 1$,
\begin{equation}\label{eq:roberts-dominating-function}
0\leq nh(X_0)^2r(X_0)^n
\leq h(X_0)^2\frac{r(X_0)}{1-r(X_0)}. 
\end{equation}
By Theorem~\ref{thm:holding-noise-bound}, $h = 0$ on the set $\{ r = 1 \}$,  and integrating the right-hand side of~\eqref{eq:roberts-dominating-function} over $\{r < 1\}$ yields a lower bound on $\sigma^2 < \infty$.  
Hence, $n h(x)^2r(x)^n$ is dominated by an integrable function. Since 
$ \lim_{n \rightarrow \infty } n r^n = 0$ for  each $r \in[0, 1)$, dominated convergence yields~\eqref{eq:roberts-limit-zero}. In particular, Roberts' condition cannot hold. 

If $\delta = 1$, then $r = 1$ everywhere and by Theorem~\ref{thm:holding-noise-bound}, $h \equiv 0$ $\pi$-almost everywhere. 
If $0 < \delta < 1$ and $r(x) \geq \delta$ for every $x$, then $r(x)/(1 - r(x)) \geq \delta / (1 - \delta)$, and thus the lower bound  on $\sigma^2$ given in Theorem~\ref{thm:holding-noise-bound} yields  
\begin{equation}
        \pi(h^2)   \leq \frac{1 - \delta}{\delta}  \int_{ \{r < 1 \} } h(x)^2 \frac{r(x)}{1 - r(x)} \pi(d x)  \leq \frac{1 - \delta}{\delta}  \sigma^2. 
\end{equation}
By Corollary~\ref{coro:P2}, this implies $A < \infty$.  
\end{proof} 

\subsection{HR Problems 1 and 7: Counterexamples with Undefined $B$}
\label{sec:P1}
\begin{HRproblem}{1}
 If the Markov chain is ergodic, but not necessarily reversible nor geometrically ergodic nor uniformly ergodic, does it necessarily follow that $A = B$ (allowing that they may both be infinite)? What if we assume that $P = P_1 P_2$ where each $P_i$ is reversible?      
\end{HRproblem}

\begin{HRproblem}{7 (restricted version)} 
  If $P$ is reversible and ergodic  and Roberts' condition holds,  does it necessarily follow that $A = B = \infty$? 
\end{HRproblem}

For HR Problem~1, we interpret the question as asking whether the conclusion holds for $h \in L^2(\pi)$, since Example~11 of~\citet{HaggstromRosenthal2007} already provides a counterexample where $h \notin L^2(\pi)$, $A = \infty$, and $B$ is undefined due to $\gamma_0 = \infty, \gamma_1 = -\infty$. 
For HR Problem~7, here we focus on the most challenging part stated in the restricted version above. For the full statement and its solution, see Section~\ref{appx:HR7} in the Appendices. 

\citet{HaggstromRosenthal2007} proved that for $h \in L^2(\pi)$, $A=B$  if the chain is reversible and ergodic, or if the chain is uniformly ergodic (see their Remark 8). In particular, we have $A = B$ on any finite state space since ergodicity would imply uniform ergodicity and $h$ is always square-integrable. 
Also, by~\eqref{eq:cesaro-identity}, the convergence of $(B_m)_{m \geq 1}$ implies that $(A_n)_{n \geq 1}$ converges to the same limit. Hence, for both HR Problems~1 and~7, a counterexample must have an infinite state space and an undefined $B$. 
Further, the counterexample for HR Problem~1 must have a nonreversible kernel $P$, and that for HR Problem~7 must have $h \notin L^2(\pi)$. 
In the theorems below, we prove that both counterexamples can be constructed. 

\begin{theorem}[Solution to HR Problem 1] \label{thm:hr1-counterexample}
There exist a countable-state, irreducible, aperiodic, positive recurrent Markov transition kernel $P$ with stationary distribution $\pi$ and a function $h\in L^2(\pi)$ with $\pi(h)=0$ such that
\begin{equation}\label{eq:main-conclusion}
A = \lim_{n\to\infty}A_n=1,
\qquad
\liminf_{n\to\infty}B_n=1,
\qquad
\limsup_{n\to\infty}B_n=\infty, 
\end{equation} 
and $h$ satisfies a $\sqrt{n}$-CLT. 
Moreover, $P$ can be expressed as $P=P_1P_2$, where $P_1$ and $P_2$ are both reversible Markov transition kernels. 
\end{theorem} 

The proof of Theorem~\ref{thm:hr1-counterexample} is deferred to the next section. 

\begin{theorem}[Solution to HR Problem 7]
\label{thm:hr7-counterexample}
Let $w_k \in (0, \infty)$ and $r_k \in [0, 1)$ for each $k \geq 1$. 
There exist a reversible, irreducible, aperiodic, positive recurrent countable-state Markov chain and  centered $h\in L^1(\pi)\setminus L^2(\pi)$ such that 
\begin{equation}\label{eq:roberts-nontrivial}
 \E_\pi[ h(X_0)^2  r(X_0)^n ]  = \sum_{k \geq 1} w_k r_k^n  \text{ for every } n \geq 1,   
\end{equation}
$A = \infty$, and $B$ is undefined.  In particular, $(w_k, r_k)_{k \geq 1}$   can be chosen such that $\sum_{k \geq 1} w_k r_k <\infty$, Roberts' condition holds, and $\gamma_1 = - \infty$. 
\end{theorem}

\begin{remark}\label{rmk:P7}
The counterexample   in Theorem~\ref{thm:hr7-counterexample} is stronger than necessary to resolve HR Problem~7. 
Our construction satisfies Roberts' condition with $\E_{\pi}\bigl[h(X_0)^2r(X_0)^n\bigr] < \infty$ for every $n \geq 1$, while much simpler counterexamples can be constructed if one allows these expectations to be infinite for all large $n$.    
Besides, we make $B$ undefined through the indeterminate expression $\infty - \infty$, rather than because any individual $\gamma_k$ is undefined. 
The main idea is to introduce two classes of states, with parameters  chosen so that their intended effects do not interfere with each other. 
The ``oscillating states'' are modeled in the same way as the sign-alternating construction in Example~14 of \citet{HaggstromRosenthal2007}, and they supply the infinite second moment and the negative lag-one product. 
The ``sticky states'' have long holding times, ensuring Roberts's condition while contributing only a finite amount to the second moment. 
\end{remark}

\begin{proof}[Proof of Theorem~\ref{thm:hr7-counterexample}]
Let the state space be 
\begin{equation}
\X =\{0\}\cup\bigl\{o_{k,\varepsilon}, \, s_{k,\varepsilon}: k\geq 1,\ \varepsilon = \pm 1 \bigr\}.
\label{eq:state-space}
\end{equation}
The state $0$ is treated as the ``hub'', and the letters $o$ and $s$ stand for oscillating and sticky, respectively. 
Let $a_k = 2^k (1 + w_k)$ and define the stationary probability measure by
\begin{equation}
\pi(0)= \frac{2}{3} - \sum_{k \geq 1} \frac{w_k}{a_k^2}, \qquad
\pi(o_{k,\varepsilon})= \frac{4^{-k} }{2} ,\qquad
\pi(s_{k,\varepsilon})= \frac{w_k}{2 a_k^2}.
\label{eq:pi-definition}
\end{equation}
One can verify that $\sum_{x \in \X} \pi(x) = 1$ and $\pi(0) \geq 7/12$, since $w / (1 + w)^2 \leq  1 / 4$ for any $w > 0$ and thus $\sum_k w_k/a_k^2 \leq 1/12$.  
Define the transition probabilities at oscillating and sticky states by 
\begin{align*}
P(o_{k,\varepsilon}, \, o_{k,-\varepsilon})=\frac{1}{2},  &\qquad
P(o_{k,\varepsilon},0)=\frac{1}{2}, \\ 
P(s_{k,\varepsilon}, \, s_{k,\varepsilon})= r_k,  &\qquad
P(s_{k,\varepsilon},0)=1 - r_k.
\end{align*} 
That is, from an oscillating state, the chain either changes sign $\varepsilon$ or returns to the hub; from a sticky state, it either remains in place or returns to the hub. From the hub, put
\begin{equation*}
    P(0, x) = \frac{\pi(x) P(x, 0)}{\pi(0)}, \qquad P(0, 0) = 1 - \sum_{x \neq 0} P(0, x). 
\end{equation*}
Since $\pi(0) > 1/2$, we have $\sum_{x \neq 0} \pi(x) P (x, 0) \leq \sum_{x \neq 0} \pi(x) < \pi(0)$, ensuring $P(0, 0) > 0$. All remaining transition probabilities are zero. 
By construction, $P$ is irreducible and satisfies the detailed balance condition, and thus it is also $\pi$-reversible and positive recurrent. 
Since $P(0, 0) > 0$, it is  aperiodic.

Define $h \colon \X \rightarrow \R$ by 
\begin{equation}
h(0) =0, \qquad 
h(o_{k,\varepsilon}) =\varepsilon 2^k,\qquad
h(s_{k,\varepsilon}) =\varepsilon a_k.
\label{eq:h-definition}
\end{equation}
The sign symmetry gives $\pi(h)=0$. Also, 
\begin{equation*}
    \pi(|h|) = \sum_{k \geq 1} 2^{-k} + \sum_{k \geq 1} \frac{w_k}{a_k} \leq 2, \qquad 
    \pi(h^2) = \sum_{k \geq 1} (1 + w_k) = \infty. 
\end{equation*}
Hence, $h \in L^1(\pi)$ but $h \notin L^2(\pi)$. 
The formula given in~\eqref{eq:roberts-nontrivial} also follows since $h(x)^2 r(x)^n > 0$ only when $x$ is a sticky state. 
We always have $A = \infty$, since for $n \geq 2$, 
\begin{equation*}
    \E_\pi[ S_n^2 ] \geq \frac{1}{2} P(0, 0)^{n - 2} \sum_{k, \varepsilon} \pi(o_{k, \varepsilon})  h(o_{k,\varepsilon})^2 = \frac{1}{2} P(0, 0)^{n - 2} \sum_{k \geq 1} 1 = \infty,   
\end{equation*}
where the inequality follows by considering all paths of the form $(o_{k, \varepsilon}, 0,  \dots, 0)$. 
     
Consider $B$. We have already shown $\gamma_0 = \pi(h^2 ) = \infty$. 
To compute $\gamma_1$, we note that only self-transitions at sticky states contribute to the positive lag-one product, and only sign flips at oscillating states contribute to the negative lag-one product. Hence, with $Y = h(X_0) h(X_1)$,  
\begin{equation*}
    \E_\pi[ Y^+ ] = \sum_{k \geq 1} w_k r_k, \qquad 
    \E_\pi[ Y^- ] = \frac{1}{2} \sum_{k \geq 1} 1 = \infty. 
\end{equation*}
When $\sum_{k \geq 1} w_k r_k < \infty$, $\gamma_1 = -\infty$ and $B$ is undefined due to $\infty - \infty$; in this case, the expectation in~\eqref{eq:roberts-nontrivial} is finite for every $n \geq 1$.
When $\sum_{k \geq 1} w_k r_k = \infty$, $\gamma_1$ and thus $B$ are both undefined.  
 
Finally, consider the special case $w_k = 2^{-k}$ and $r_k = 1 - 4^{-k}$. Clearly, $\sum_{k \geq 1} w_k r_k < \infty$. 
It remains to prove that Roberts' condition also holds.   
Let $K_n:=\left\lceil\log_4(2n)\right\rceil$.  For $k\geq K_n$, one has $n4^{-k}\leq 1/2$,  and Bernoulli's inequality gives $ (1-4^{-k} )^n\geq  1/2.$ Hence, 
\begin{equation*}
       \E_\pi[ h(X_0)^2  r(X_0)^n ]  =   \sum_{k \geq 1} 2^{-k} (1 - 4^{-k})^n \geq \frac{1}{2} \sum_{k \geq K_n} 2^{-k}  = 2^{ -K_n } \geq \frac{1}{2 \sqrt{2n}}. 
\end{equation*}
So $n \E_\pi[ h(X_0)^2  r(X_0)^n ] \rightarrow \infty$, completing the proof.  
\end{proof}

\section{Solution to HR Problem~1} \label{sec:proof-P1}

\subsection{A Fourier-to-Markov Representation Theorem} \label{sec:fourier}

The main technical contribution of our solution to HR Problem~1 is a general device, presented in Theorem~\ref{thm:representation}, for realizing the Fourier coefficients of a given function as the autocovariances of a regenerative Markov chain with finite excursions.  

For an integrable $2\pi$-periodic function $v$, write
\begin{equation}\label{eq:fourier-conventions}
\Fcoef{v}{k}
=\frac{1}{2\pi}\int_{-\pi}^{\pi}v(\theta) e^{-ik\theta} \,d\theta.  
\end{equation}
For later use, define its symmetric Fourier partial sums and Fej\'er means at the origin by 
\begin{equation}\label{eq:fourier-sums-means}
\Fsum{N}{v}=\sum_{|k|\leq N}\Fcoef{v}{k},
\qquad
\Fmean{N}{v} =\sum_{|k|\leq N}\left(1-\frac{|k|}{N+1}\right)\Fcoef{v}{k}.
\end{equation} 

\begin{theorem}\label{thm:representation}
Let $f$ be a strictly positive, continuous, even, $2\pi$-periodic function.  
There exist a countable-state, irreducible, aperiodic, positive recurrent Markov chain $P$ with stationary distribution $\pi$ and a centered $h\in L^2(\pi)$  such that 
\begin{equation}\label{eq:gamma-fourier}
\gamma_k \coloneqq \E_\pi[ h(X_0) h(X_k)]   = \frac{1}{2} \Fcoef{f}{k},  \qquad k\geq0, 
\end{equation}
and $h$ satisfies a $\sqrt{n}$-CLT. Moreover, $P$ can be expressed as $P=P_1P_2$,  where each $P_i$ is reversible with respect to $\pi$.  
\end{theorem}

The proof of Theorem~\ref{thm:representation} separates into an analytic component and a probabilistic component. 
The key analytic decomposition result is presented in Lemma~\ref{lem:autocorrelation-decomposition} below, which shows  that $\Fcoef{f}{k}$  can be expressed as a mixture of finite-length lag-$k$ autocorrelation sums. 
Then,  given any square-summable sequence of vectors $(a^{(q)})_{q \geq 1}$ as in Lemma~\ref{lem:autocorrelation-decomposition},  we construct an excursion-based Markov chain such that its autocovariances coincide with the right-hand side of~\eqref{eq:finite-real-decomp-Fk} up to a constant factor. 
This chain is very useful for studying the behavior of variance quantities $A$ and $B$, since when~\eqref{eq:gamma-fourier} holds for even $f$,  $\Fcoef{f}{k} = \Fcoef{f}{-k}$, and by \eqref{eq:cesaro-identity},  
\begin{equation}\label{eq:BN-as-fourier-sum}
A_n=\frac12\Fmean{n-1}{f}, \qquad B_n=\frac12\Fsum{n}{f}. 
\end{equation}  
 
\begin{lemma}\label{lem:autocorrelation-decomposition}
Let $f$ be a strictly positive, continuous, even, $2\pi$-periodic function. Then $\Fcoef{f}{0}<\infty$, and there exist finite real vectors  $a^{(q)}=(a_{q,0},\ldots,a_{q,\ell_q-1})$ for $q \geq 1$ such that 
\begin{equation}\label{eq:finite-real-decomp-Fk}
 f(\theta) =\sum_{q=1}^{\infty} \left|\sum_{j=0}^{\ell_q-1}a_{q,j} e^{i j \theta}\right|^2, \qquad   
 \Fcoef{f}{k} = \sum_{q=1}^{\infty}\sum_{j=0}^{\ell_q-1-k}a_{q,j}a_{q,j+k} \;    \text{ for $k \geq 0$,} 
\end{equation} 
where the inner sum is zero when $k\geq\ell_q$.    
\end{lemma}

\begin{proof}
    See Section~\ref{appx:P1} in the Appendices. 
\end{proof}

\begin{remark} 
Theorem~\ref{thm:representation} is similar in spirit to Herglotz's theorem, which can be used to show that the autocovariances of every stationary Gaussian process can be realized through the Fourier transform of  a finite symmetric positive measure~\citep{PeligradWu2010}.  
At an analytical level, there exist results similar to Lemma~\ref{lem:autocorrelation-decomposition} which  decompose positive continuous functions into sums of positive trigonometric polynomials; see, for example,~\citet[Theorem~2.4.2 and Exercise~2.4.3]{Rudin1969} and the explicit discussion in \citet[Section~4.1]{ChenYuan2023}. 
\end{remark}

\begin{proof}[Proof of Theorem~\ref{thm:representation}]
By Lemma~\ref{lem:autocorrelation-decomposition}, we can choose $a^{(q)}=(a_{q,0},\ldots,a_{q,\ell_q-1})$ for $q\geq 1$ such that~\eqref{eq:finite-real-decomp-Fk} holds. 
We now construct a Markov chain $P$ satisfying the given assumptions such that 
\begin{equation}\label{eq:general-excursion-covariance}
\gamma_k =\frac12\sum_{q=1}^{\infty}\sum_{j=0}^{\ell_q-1-k}a_{q,j}a_{q,j+k},
\qquad k\geq0. 
\end{equation}  

Let the state space be 
\begin{equation}\label{eq:general-state-space}
\X=\{0\}\cup\bigl\{(q,\varepsilon,j):q\geq1,\ \varepsilon = \pm 1,\ 0\leq j<\ell_q\bigr\}.
\end{equation}
We let $0$ be the ``regeneration state'' of the chain. From $0$, the chain stays at $0$ with probability $1/2$; otherwise, it selects a block $q \geq 1$ with  probability $\rho_q$ and an independent symmetric sign $\varepsilon$,  deterministically traverses the states  $(q, \varepsilon, 0), \ldots, (q, \varepsilon, \ell_q - 1)$ and then returns to $0$. The transition matrix can be expressed by 
\begin{equation}\label{eq:general-transition-kernel}
\begin{array}{ll}
  P(0,0)=\frac{1}{2},    &  \\[5pt]
  P\bigl(0, \, (q,\varepsilon,0)\bigr)=\frac{\rho_q}{4}    &  \text{ for } q \geq 1, \quad \varepsilon = \pm 1,  \\[5pt]
  P\bigl((q,\varepsilon,j), \, (q,\varepsilon,j+1)\bigr)=1 \quad & 
  \text{ for }
  q \geq 1, \quad  \varepsilon = \pm 1, \quad 0\leq j<\ell_q-1, \\[5pt] 
  P\bigl((q,\varepsilon,\ell_q-1), \, 0 \bigr)=1 & 
  \text{ for } q \geq 1, \quad  \varepsilon = \pm 1. 
\end{array}  
\end{equation} 
The dynamics of the chain is visualized in Figure~\ref{fig:regenerative-excursion-dynamics}.  We let $\rho_q$ be  
\begin{equation}
\rho_q=\frac{2^{-q}}{c_\rho\ell_q}, \qquad 
    c_\rho=\sum_{q=1}^{\infty}\frac{2^{-q}}{\ell_q} < \infty. 
\end{equation} 
So $\sum_{q \geq 1} \rho_q = 1$. 
Every state communicates with $0$, and the self-loop at $0$ makes the chain aperiodic. Further, a direct calculation verifies that  $P$ has stationary distribution 
\begin{equation}\label{eq:general-stationary-distribution}
\pi(0)=\alpha,
\qquad
\pi(q,\varepsilon,j)=\frac{\alpha\rho_q}{4}, \qquad 
\alpha=\left(1+\frac{1}{2} \sum_{q=1}^{\infty}\rho_q\ell_q \right)^{-1} > 0. 
\end{equation}  
Hence the chain is also positive recurrent. 
Define the function $h$ by 
\begin{equation}\label{eq:general-observable}
h(0)=0,
\qquad
h(q,\varepsilon,j)
=\varepsilon\frac{a_{q,j}}{\sqrt{\alpha\rho_q}}.
\end{equation} 
The symmetric sign $\varepsilon$ gives $\pi(h)=0$, while, by Lemma~\ref{lem:autocorrelation-decomposition}, 
\begin{align}\label{eq:general-L2}
\pi(h^2)
 =\sum_{q=1}^{\infty}\sum_{\varepsilon\in\{-1,1\}}
\sum_{j=0}^{\ell_q-1}
\frac{\alpha\rho_q}{4}\frac{a_{q,j}^2}{\alpha\rho_q}  
 =\frac12\sum_{q=1}^{\infty}\sum_{j=0}^{\ell_q-1}a_{q,j}^2 
 =\frac{1}{2} \Fcoef{f}{0} <\infty.
\end{align}  
Starting from $0$, every future nonzero excursion carries a fresh symmetric sign, and therefore $\E_0[h(X_n)]=0$ for every $n \geq 0$.  
If the chain starts at $(q,\varepsilon,j)$, its motion is deterministic until the current excursion ends; after the return to $0$, the conditional expectation of future observations is again zero. Thus, 
\begin{equation}\label{eq:general-Pkh}
(P^kh)(q,\varepsilon,j)
=\one_{\{j+k<\ell_q\}}\varepsilon
\frac{a_{q,j+k}}{\sqrt{\alpha\rho_q}}.
\end{equation}
Using stationarity and~\eqref{eq:general-stationary-distribution} and~\eqref{eq:general-Pkh}, we obtain
\begin{align}\label{eq:general-covariance-calculation}
\gamma_k
&= \E_\pi [ h(X_0) (P^k h)(X_0) ] \\
&=\sum_{q=1}^{\infty}\sum_{\varepsilon\in\{-1,1\}}
\sum_{j=0}^{\ell_q-1-k}
\frac{\alpha\rho_q}{4}
\left(\varepsilon\frac{a_{q,j}}{\sqrt{\alpha\rho_q}}\right)
\left(\varepsilon\frac{a_{q,j+k}}{\sqrt{\alpha\rho_q}}\right)\\
&=\frac12\sum_{q=1}^{\infty}\sum_{j=0}^{\ell_q-1-k}a_{q,j}a_{q,j+k}. 
\end{align} 
By~\eqref{eq:BN-as-fourier-sum} and Fej\'er's theorem (see Lemma~\ref{lem:fejer-theorem}), 
\begin{equation}
    A_n  = \frac12\Fmean{n-1}{f} \rightarrow \frac{1}{2} f(0) < \infty. 
\end{equation}
Hence, $\sup_{n \geq 1} A_n = \sup_{n \geq 1} \E_\pi[ S_n^2 / n ]< \infty$, which implies that $(S_n / \sqrt{n})_{n \geq 1}$ is tight. 
By Theorem~\ref{th:poisson}, a $\sqrt{n}$-CLT holds.

It remains to show that $P$ admits the decomposition $P = P_1 P_2$ where both $P_1$ and $P_2$ are reversible  (we use the convention that $P_1P_2$ first applies $P_1$ and then $P_2$).  The key observation is that each deterministic forward step in a block with $\ell$ states can be factorized into two reversible reflections:  $j \rightarrow j + 1$ is equivalent to  $j \rightarrow  k$  with $k = \ell - 1 - j $ and then $k \rightarrow \ell - k$, with all operations performed modulo $\ell$.  
Explicitly, define the transition matrix $P_1$ by 
\begin{equation*}
P_1(0, 0) = 1, \qquad P_1( (q, \varepsilon, j), \,  (q, \varepsilon, \ell_q - 1 - j) ) = 1.
\end{equation*} 
All other transitions have probability zero. Then $P_1$ is an involution (i.e., $P_1 P_1 = I$) with cycles whose positions have equal $\pi$-mass, and thus $P_1$ is $\pi$-reversible. 
Define $P_2 = P_1 P$ so that $P = (P_1 P_1) P = P_1 P_2$.  
Then, 
\begin{equation*} 
\begin{array}{ll}
  P_2(0,0)=\frac{1}{2},    &  \\[5pt]
  P_2\bigl(0, \, (q,\varepsilon,0)\bigr)=\frac{\rho_q}{4}    &  \text{ for } q \geq 1, \quad \varepsilon = \pm 1,  \\[5pt]
  P_2\bigl((q,\varepsilon,j), \, (q,\varepsilon, \ell_q - j )\bigr)=1 \quad & 
  \text{ for } 
  q \geq 1, \quad  \varepsilon = \pm 1, \quad 0 < j \leq \ell_q-1, \\[5pt] 
  P_2 \bigl((q,\varepsilon, 0), \, 0\bigr)=1 & 
  \text{ for } q \geq 1, \quad  \varepsilon = \pm 1. 
\end{array}  
\end{equation*}   
By checking the detailed balance condition at states $0$ and $(q, \varepsilon, 0)$,  we can verify that $P_2$ is also $\pi$-reversible, thus completing the proof. 
\end{proof}
  
\suppressfloats[h]
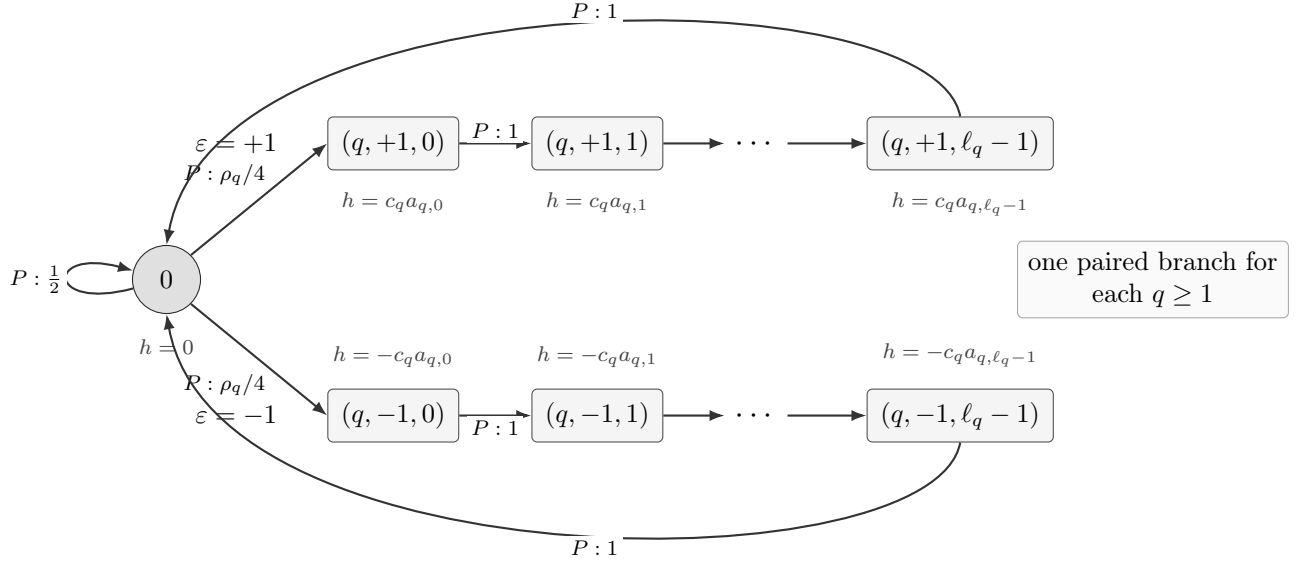
\begin{figure}[h]
\centering
\resizebox{\textwidth}{!}{\begin{tikzpicture}[x=1cm,y=1cm]
  \node[mchub] (hub) at (0.65,0) {$0$};
  \node[mcvalue,below=2mm of hub] {$h=0$};
  \draw[mcedge] (hub) to[loop left,min distance=12mm]
    node[mcprob,left] {$ \tfrac12$} (hub);

  \node[font=\small\bfseries,anchor=east] at (2.25,1.8)
    { };
  \node[mcstate] (plus-zero) at (3.65,1.8) {$(q,+1,0)$};
  \node[mcstate] (plus-one) at (6.35,1.8) {$(q,+1,1)$};
  \node[font=\large] (plus-dots) at (8.45,1.8) {$\cdots$};
  \node[mcstate] (plus-last) at (11.15,1.8)
    {$(q,+1,\ell_q-1)$};
  \node[mcvalue,below=1.5mm of plus-zero]
    {$h=c_q a_{q,0}$};
  \node[mcvalue,below=1.5mm of plus-one]
    {$h=c_q a_{q,1}$};
  \node[mcvalue,below=1.5mm of plus-last]
    {$h=c_q a_{q,\ell_q-1}$};

  \node[font=\small\bfseries,anchor=east] at (2.25,-1.8)
    { };
  \node[mcstate] (minus-zero) at (3.65,-1.8) {$(q,-1,0)$};
  \node[mcstate] (minus-one) at (6.35,-1.8) {$(q,-1,1)$};
  \node[font=\large] (minus-dots) at (8.45,-1.8) {$\cdots$};
  \node[mcstate] (minus-last) at (11.15,-1.8)
    {$(q,-1,\ell_q-1)$};
  \node[mcvalue,above=1.5mm of minus-zero]
    {$h=-c_q a_{q,0}$};
  \node[mcvalue,above=1.5mm of minus-one]
    {$h=-c_q a_{q,1}$};
  \node[mcvalue,above=1.5mm of minus-last]
    {$h=-c_q a_{q,\ell_q-1}$};

  \draw[mcedge] (hub.north east) --
    node[mcprob,above left,pos=0.58] { $\rho_q/4$ } (plus-zero.west);
  \draw[mcedge] (plus-zero) -- node[mcprob,above] { 1 } (plus-one);
  \draw[mcedge] (plus-one) -- (plus-dots);
  \draw[mcedge] (plus-dots) -- (plus-last);
  \draw[mcedge] (plus-last.north) .. controls (10.8,4.05) and (1.3,4.05) ..
    node[mcprob,above,pos=0.48] {$ 1$} (hub.north);

  \draw[mcedge] (hub.south east) --
    node[mcprob,below left,pos=0.58] { $\rho_q/4$ } (minus-zero.west);
  \draw[mcedge] (minus-zero) -- node[mcprob,below] { $1$ } (minus-one);
  \draw[mcedge] (minus-one) -- (minus-dots);
  \draw[mcedge] (minus-dots) -- (minus-last);
  \draw[mcedge] (minus-last.south) .. controls (10.8,-4.05) and (1.3,-4.05) ..
    node[mcprob,below,pos=0.48] {$ 1$} (hub.south);

  \node[draw=black!35,rounded corners=2pt,fill=black!2,
        font=\small,align=center,inner sep=4pt] at (13.7,0)
    {one paired branch for\\each $q\geq1$};
\end{tikzpicture}}
\caption{Dynamics of the basic excursion chain.  Here \(c_q=(\alpha\rho_q)^{-1/2}\). }
\label{fig:regenerative-excursion-dynamics}
\end{figure}

The finite excursion architecture used in our proof of Theorem~\ref{thm:representation} is a  classical regenerative construction~\citep{Smith1955}. It was shown in~\citet{Glynn1997}  that covariance functions and spectral quantities of stationary regenerative processes can be represented in terms of quantities defined over individual excursions. This excursion-chain idea also appears in \citet[Example~11 and Remark~13]{HaggstromRosenthal2007}, where each visit to \(0\) initiates a deterministic finite excursion.  Our construction allows an arbitrary finite sequence $a^{(q)}$ on each excursion and introduces an independent symmetric sign so that correlations across different excursions vanish.  
 
For the factorized construction $P = P_1 P_2$ used in our proof, 
the transition kernels $P_1, P_2$ are reversible but are not conditional-expectation projections. Consequently, although we resolve HR Problem~1, our result does not imply the same conclusion for two-component Gibbs samplers, which provided one of the motivations for this problem. 
 
\subsection{Proof of Theorem~\ref{thm:hr1-counterexample} }  

\begin{proof}  
Let $f$ be a strictly positive, continuous, even, $2\pi$-periodic function.  By Theorem~\ref{thm:representation}, we can construct a countable space and ergodic $P$ and a square-integrable function $h$ such that $\gamma_k = \Fcoef{f}{k}/2$ and $h$ satisfies a $\sqrt{n}$-CLT. Further,  $P = P_1 P_2$ with $P_1, P_2$ reversible with respect to $\pi$. 
By~\eqref{eq:BN-as-fourier-sum},  we have $2A_n = \Fmean{n-1}{f}$ and $2 B_n = \Fsum{n}{f}$. By Fej\'er's theorem, $\Fmean{N}{f}\longrightarrow f(0)$, and thus $A = f(0) / 2$ exists.  
Meanwhile, ordinary Fourier partial sums need not converge for continuous functions. 
Therefore,  it only remains to find a function $f$  such that  
\begin{equation}\label{eq:target-fourier-properties}
f(0)=2, \qquad
\liminf_{N\to\infty}\Fsum{N}{f}=2, \qquad
\limsup_{N\to\infty}\Fsum{N}{f}=\infty. 
\end{equation} 
Here we simply use a classical example due to Fej\'er: 
\begin{equation} 
g(\theta) =\sum_{r=1}^{\infty}
\frac{1}{r^2}
\sin\left(\frac{2^{r^3}+1}{2}|\theta|\right),
\qquad 
f(\theta)=2+g(\theta), \qquad
\theta \in [-\pi, \pi]. 
\end{equation}
Note that $f$ is strictly positive since 
\[
f(\theta)\geq 2-\sum_{r=1}^{\infty}\frac{1}{r^2}
=2-\frac{\pi^2}{6}>0. 
\] 
The well-known properties of the function $g$, which we review in Lemma~\ref{lem:fejer-example}, imply that $f$ satisfies~\eqref{eq:target-fourier-properties}. 
Hence, $B_n$ does not converge, while $A_n \rightarrow 1$.  
\end{proof} 

\section{Concluding Remarks} \label{sec:disc}
In this paper, we have resolved all the seven open problems posed by~\citet{HaggstromRosenthal2007}. There are several directions for future research. First, although our counterexample to HR Problem~1 shows that \(A=B\) may fail when \(P=P_1P_2\) with both \(P_1\) and \(P_2\) reversible, it is unclear such counterexamples still exist for a genuine two-component Gibbs sampler (or more generally, for any nonreversible MCMC algorithms used in practice).  
Second, our Poisson characterization establishes a connection between the tightness of $( S_n/\sqrt{n})_{n \geq 1}$ and the existence of a martingale decomposition with square-integrable differences, without requiring \(h\in L^2(\pi)\). A natural extension is to investigate   whether a functional CLT can be established in this setting, particularly when \(h\notin L^2(\pi)\). 
Third, for nonreversible chains with $h \in L^2(\pi)$, an open question is what the necessary and sufficient condition is for $A = B$ and for $A = \sigma^2$. 
Finally, it would be worthwhile to investigate whether quantitative approximations of the Poisson solution based on Theorem~\ref{th:poisson} can lead to useful methods for estimating the CLT variance in practice. 

\section*{Acknowledgment}
JSR was supported in part by NSERC of Canada discovery grant RGPIN-2019-04142. QZ was supported in part by NSF through grants DMS-2311307 and DMS-2245591.

\bigskip 
\bigskip 
\noindent {\LARGE\bfseries Appendices}

\appendix
\numberwithin{equation}{section}
\numberwithin{figure}{section}
\numberwithin{table}{section}

\section{Auxiliary Results for Theorem~\ref{th:poisson} }\label{appx:th-poisson}

\subsection{Auxiliary Lemmas} 
\nocite{billingsley1999convergence}
\begin{lemma}[Billingsley, 1999, Theorem 18.3]\label{lem:mart-CLT}
Let $(\xi_n)_{n \geq 0}$ be a stationary, ergodic martingale-difference sequence with respect to its natural filtration. If $s^2 \coloneqq \E[\xi_0^2] < \infty$,  then 
\begin{equation*}
   \frac{1}{\sqrt{n}} \sum_{k=0}^{n-1} \xi_k \Longrightarrow \Normal(0,  s^2). 
\end{equation*}
\end{lemma}

\begin{lemma}\label{lem:random-index}
Let $(U_n)_{n \geq 1}$ be  tight. 
For $0 < p < 1$, let $N_p \in \{1, 2, \dots \}$ denote a geometric random variable with success probability $p$ that is independent of $(U_n)_{n \geq 1}$; in particular, $\E[ N_p] = 1 / p$. 
Then the family 
$$\left(\sqrt{p N_p} \, U_{N_p}\right)_{0 < p < 1}$$ 
is also tight, and their characteristic functions are equicontinuous at zero.  
\end{lemma} 

\begin{proof}
 For every $\epsilon, R > 0$, we have 
\begin{align*}
    \Prob\left( \left|\sqrt{p N_p} \, U_{N_p}\right| > R \right) &\leq  \Prob( | p N_p| > \epsilon^{-1} ) + \Prob\left( |U_{N_p}| > R\sqrt{\epsilon}\right)  \\
    &\leq \epsilon + \sup_{n \geq 1} \Prob\left( |U_n| >  R \sqrt{\epsilon} \right),
\end{align*}
since $\E[p N_p] = 1$. Since $(U_n)_{n \geq 1}$ is tight, for each fixed $\epsilon > 0$, we can choose $R$ sufficiently large so that the supremum term is bounded by $\epsilon$, and thus $  \Prob( |\sqrt{p N_p} \, U_{N_p}| > R ) \leq 2 \epsilon$. This proves that $(\sqrt{p N_p} \, U_{N_p})_{0 < p < 1}$ is also tight. 
By~\citet[Lemma 4.2]{kallenberg1997foundations}, a family of random variables is tight if and only if their characteristic functions are equicontinuous at zero, which completes the proof.  
\end{proof}

\subsection{Proof of Lemma~\ref{lem:neveu-inverse} }\label{appx:neveu}

\begin{proof}
Since there exists a set $G$ with $\pi(G) = 1$ such that $P$ restricted to $G$ is positive Harris recurrent and the chain $(X_n)_{n \geq 0}$ remains in $G$ almost surely, we can always work with this restriction and extend functions arbitrarily outside $G$. 
By~\citet[Proposition 4.4 and Theorem 5.1]{Neveu1972}, there exist a strictly positive measurable function $a \leq 1$ and a positive kernel $W$ such that, with $b=\pi(a) \leq 1$, 
\begin{align}
     (a\pi) (W f) &=  \frac{1-b}{b}  \pi (f), \label{eq:neveu-weighted-kernel} \\
     Pf + (PW) f &= Wf + \frac{ \pi(f)  }{b} P a, \label{eq:neveu-left-kernel} \\
    Pf + (WP)f &= Wf + \frac{ (a\pi) (P f) }{b} , \label{eq:neveu-right-kernel}
\end{align}
for every bounded measurable function $f$. 

Set $R=I+W$. For every $f\in L^1(\pi)$, $a \in (0, 1]$ and~\eqref{eq:neveu-weighted-kernel} give
\begin{align}
    \norm{Rf}_{L^1(a\pi)}
    &\leq (a \pi) ( |f| )+(a\pi) ( W|f| ) \notag \\
    &= \pi(a|f|)+ \frac{1-b}{b} \pi(|f|) 
    \leq \frac{1}{b}\norm{f}_{L^1(\pi)}. \label{eq:neveu-R-bound}
\end{align}
In particular, $W|f|$ and $Rf$ are finite $\pi$-almost surely. Equation~\eqref{eq:neveu-left-kernel} also gives
\begin{equation}\label{eq:neveu-row-integrability}
    P|Rf| \leq (PR)|f|
    = W|f|+ \frac{\pi(|f|)}{b}Pa  < \infty
    \quad \text{$\pi$-almost surely}.
\end{equation}
Hence the three identities~\eqref{eq:neveu-weighted-kernel},~\eqref{eq:neveu-left-kernel} and~\eqref{eq:neveu-right-kernel} can also be applied to signed integrable functions, since all the terms involved are finite almost surely. 

Define  $\rho(dx)=a(x)\pi(dx)/b$. 
For $f\in L^1_0(\pi)$, define the linear map 
\[
    \mathcal U f = Rf-\rho(f).
\]
Equation~\eqref{eq:neveu-left-kernel} and $\pi(f) = 0$ yield
\begin{align*}
    (I-P)(\mathcal U f) = (I-P)( R f) = R f - W f - \frac{\pi(f)}{b} P a = f. 
\end{align*}
Also,~\eqref{eq:neveu-weighted-kernel} and $\pi(f) = 0$ give 
\begin{align*}
    \rho(\mathcal U f) = \rho(Rf) - \rho(f) = \rho(W f) = \frac{ (a \pi) (W f) }{b} = \frac{1 - b}{b^2}\pi(f) = 0. 
\end{align*}
By~\eqref{eq:neveu-row-integrability}, $P |\mathcal U f| < \infty$, which proves~\eqref{eq:neveu-poisson-inverse}. 

Inequality~\eqref{eq:neveu-weighted-bound} follows from~\eqref{eq:neveu-R-bound}: 
\begin{align*}
    \norm{\mathcal U f}_{L^1(a\pi)}
    &\leq \norm{Rf}_{L^1(a\pi)}+b|\rho(f)| \\
    &\leq \frac{1}{b}\norm{f}_{L^1(\pi)}+|\pi(af)|
    \leq \left(1+\frac{1}{b}\right)\norm{f}_{L^1(\pi)}. 
\end{align*} 

For $v\in L^1(\pi)$, \eqref{eq:neveu-right-kernel} yields 
\[
    (R P) v= W v + \frac{ (a \pi) (P v) }{b} = W v + \rho(Pv).
\]
Further, $(I-P)v\in L^1_0(\pi)$ since $ \pi = \pi P$. So it follows that 
\begin{align*}
    \mathcal U [(I-P)v] = R[(I-P)v] - \rho[(I-P)v]= v -\rho(Pv) -  \rho[(I-P)v]  = v - \rho(v). 
\end{align*}

Finally, if $f_j \rightarrow f$ in $L^1(\pi)$, we have $\mathcal U f_j\to\mathcal U f$ in $L^1(a\pi)$ by~\eqref{eq:neveu-weighted-bound}. 
Since $\pi$ and $a\pi$ are equivalent measures, we also have the  convergence in $\pi$-probability.
\end{proof}

\section{Auxiliary Results for Theorem~\ref{thm:representation} }\label{appx:P1}

\subsection{Fej\'er's Theorem and Example}

Recall that for an integrable $2\pi$-periodic function $v$, we define 
\begin{align*} 
\Fcoef{v}{k}
 =\frac{1}{2\pi}\int_{-\pi}^{\pi}v(\theta) e^{-ik\theta} \,d\theta.  
\end{align*}
 
\begin{lemma}[Fej\'er's theorem]\label{lem:fejer-theorem}
Let $v$ be an integrable $2\pi$-periodic function, and define
\begin{equation} 
(\Fm_N v)(x) =\sum_{|k|\leq N}\left(1-\frac{|k|}{N+1}\right)\Fcoef{v}{k} e^{ikx}.
\end{equation}  
If $v$ is continuous at $x$, then $(\Fm_N v)(x) \longrightarrow v(x)$.  
If $v$ is continuous everywhere, then the convergence is uniform.
Moreover, if $v \geq0$, then $\Fm_N v \geq0$ for every $N$.
\end{lemma}

\begin{lemma}[Fej\'er's example] \label{lem:fejer-example}
    Let $g \colon [-\pi, \pi] \rightarrow \R$ be defined by 
    \begin{equation} 
        g(\theta) =\sum_{j=1}^{\infty} \frac{1}{j^2}
            \sin\left(\frac{2^{j^3}+1}{2}|\theta|\right). 
    \end{equation}
    Define $\Fcoef{g}{k}= (2\pi)^{-1} \int_{-\pi}^{\pi}g(\theta)\exp(-ik\theta)\,d\theta$, and 
    \begin{equation} 
    \Fsum{N}{g}=\sum_{|k|\leq N}\Fcoef{g}{k},
        \qquad
    \Fmean{N}{g}
    =\sum_{|k|\leq N}\left(1-\frac{|k|}{N+1}\right)\Fcoef{g}{k}.
    \end{equation}
    Then, $g$ is continuous with $g(0) = 0$,  $\Fsum{N}{g}\geq0 $ for every $N \geq 0$, and
    \begin{equation}
        \limsup_{N\to\infty}\Fsum{N}{g}=\infty, \qquad 
        \liminf_{N\to\infty}\Fsum{N}{g}=0, \qquad 
        \lim_{N \to \infty} \Fmean{N}{g} = 0. 
    \end{equation}
\end{lemma}

\begin{proof}    
This is Fej\'er's classical example of a continuous function whose Fourier series diverges at $0$~\citep[Chapter 8.8]{duren2012invitation}. 
The series defining $g$ converges uniformly since
\begin{equation}
   | g(\theta) | \leq \sum_{j=1}^\infty \frac{1}{j^2} = \frac{\pi^2}{6}.  
\end{equation}
It is shown in~\citet[Chapter 8.8]{duren2012invitation} that $\limsup_{N\to\infty}\Fsum{N}{g}=\infty$ and 
\begin{equation}
    \Fsum{N}{g} \geq 0, \text{ for every } N \geq 1. 
\end{equation}
On the other hand, Fej\'er's theorem and $g(0)=0$ imply $\Fmean{N}{g} \rightarrow 0$.    Together with the nonnegativity of $\Fsum{N}{g}$, this forces $\liminf_{N\to\infty}\Fsum{N}{g}=0.$ 
\end{proof}

\subsection{Proof of Lemma~\ref{lem:autocorrelation-decomposition}}

\begin{proof}  
We first show that $f$ can be presented as 
\begin{equation}\label{eq:uniform-positive-series}
f(\theta)=\sum_{m=1}^{\infty}T_m^2(\theta), \qquad T_m(\theta) = \sum_{k = - d_m}^{d_m} c_{m, k} e^{i k \theta}, 
\end{equation}
where the convergence holds uniformly and every $T_m$ is an even trigonometric polynomial so that $c_{m, k} = c_{m, -k} \in \R$. 
For each $m \geq 1$,  let
\begin{equation}
    r_m = f -  \sum_{k=1}^{m} T^2_k
\end{equation}
denote the residual after the $m$-th approximation, which is clearly $2\pi$-periodic. Set $r_0=f$. 

We  construct $r_m$ inductively. 
Suppose  that $r_{m-1}$ is strictly positive, continuous, and even. 
By applying Fej\'er's theorem (see Lemma~\ref{lem:fejer-theorem}) to the continuous function $\sqrt{r_{m - 1}  / 2}$,  we find that there exists a sufficiently large $n_m$  such that 
\begin{equation}\label{eq:fejer-approximation-choice}
 \left\|    \left( \Fm_{n_m} \sqrt{ \frac{ r_{m-1} }{2} } \right)^2  - \frac{ r_{m-1} }{2} \right\|_\infty 
<\frac{1}{4}   \inf_{\theta\in[-\pi,\pi]}r_{m-1}(\theta). 
\end{equation}  
The right-hand side is strictly positive, since $\sqrt{r_{m - 1}  / 2}$ is  strictly positive and continuous on $[-\pi, \pi]$.  
Now define 
\begin{equation}\label{eq:iterative-positive-decomposition}
T_m=  \Fm_{n_m} \sqrt{ \frac{ r_{m-1} }{2} }, \qquad r_m = r_{m - 1} - T_m^2. 
\end{equation}
Since $r_{m-1}$ is even,  $\sqrt{r_{m-1} / 2}$ is also even with real Fourier coefficients. Therefore, by construction, $T_m$ is an even trigonometric polynomial function which can be expressed in the form given in~\eqref{eq:uniform-positive-series} with $c_{m, k} = c_{m, -k} \in \R$. 
The approximation bound implies
\begin{equation}\label{eq:iterative-geometric-bounds}
\frac14r_{m-1}\leq T_m^2\leq\frac34r_{m-1},
\qquad
\frac14r_{m-1}\leq r_m \leq\frac34r_{m-1}.
\end{equation} 
Hence the new residual $r_m$ remains strictly positive, and it decreases geometrically: 
\begin{equation}\label{eq:residual-geometric-decay}
\lVert r_m\rVert_\infty
\leq\left(\frac34\right)^m\lVert f\rVert_\infty.
\end{equation} 
This proves that~\eqref{eq:uniform-positive-series} holds uniformly.
  
To prove the lemma using~\eqref{eq:uniform-positive-series}, we define the ordinary polynomial
\begin{equation}\label{eq:def-Phi-m}
    \Phi_m(z)=\sum_{k= -d_m}^{d_m} c_{m, k} z^{k + d_m} \eqqcolon \sum_{j = 0}^{\ell_m - 1} a_{m, j} z^j,
\end{equation}
where $\ell_m = 2 d_m + 1$ and $a_{m, j} = c_{m, j - d_m}$. The function $\Phi_m$ has real coefficients and satisfies $\Phi_m(e^{i \theta}) = e^{i d_m \theta} T_m(\theta)$. Since $|e^{i d_m \theta}| = 1$, 
\begin{equation*}
   f(\theta) = \sum_{m=1}^\infty T_m^2(\theta) = \sum_{m=1}^\infty  \left|  \Phi_m(e^{i\theta}) \right|^2. 
\end{equation*}  
The claimed expression for $f(\theta)$  follows from  the second expression for $\Phi_m$ given in~\eqref{eq:def-Phi-m}. 
Taking Fourier coefficients  and performing a direct expansion yield the claimed expression for $\Fcoef{f}{k}$.  
Finally, $\Fcoef{f}{0} < \infty$ since $f$ is continuous and thus integrable on $[-\pi, \pi]$. 
\end{proof}

\section{HR Problem~7} \label{appx:HR7}
Here we present the original statement of HR Problem~7 and the corresponding solution. 

\begin{HRproblem}{7} 
   Can the condition $\pi(h^2) < \infty$ be dropped from Corollaries 6 and 7?   
\end{HRproblem}

For readers' convenience,  we include the relevant corollary statements from~\citet{HaggstromRosenthal2007}. 

\medskip 
\noindent\textbf{HR Corollary 6.}
\textit{If $P$ is reversible and ergodic, and $\pi(h^2) < \infty$, and any one of $A, B,$ and $C$ is finite, then a $\sqrt{n}$-CLT exists for $h$, with $\sigma^2 = A = B = C < \infty$.}

\medskip 
\noindent\textbf{HR Corollary 7.}
\textit{If $P$ is reversible and ergodic, and $\lim_{n \rightarrow \infty} n \E_\pi[ h^2(X_0) r(X_0)^n ] = \infty$ where $\pi(h^2) < \infty$, then $A = B = C = \infty$.}

\medskip 

For HR Corollary~6, the condition $\pi(h^2) < \infty$ can be removed. 
The spectral quantity $C$ is defined only when $h \in L^2(\pi)$, and $B < \infty$ already implies $\gamma_0 = \pi(h^2) < \infty$. 
By Lemma~\ref{lem:inf-A-h2}, $A < \infty$ also implies $h \in L^2(\pi)$. 

For HR Corollary~7, Lemma~\ref{lem:inf-A-roberts} shows that Roberts' condition forces $A = \infty$. Since $C$ is not defined outside $L^2(\pi)$, the remaining question is whether Roberts' condition also forces $B = \infty$. 
By Theorem~\ref{thm:hr7-counterexample}, the condition $\pi(h^2) < \infty$ cannot be removed.

\end{document}